\documentclass{article}
\usepackage{graphicx} 
\usepackage{amsmath}
\usepackage{hyperref}
\usepackage{cleveref}
\usepackage[margin=3cm]{geometry}
\usepackage[caption=false]{subfig}
\usepackage{todonotes}
\usepackage{amsthm}
\usepackage{amssymb}

\usepackage[backend=biber,
            style=nature,
            sortcites=true,
            giveninits=true, 
            useprefix=true,
            sorting=nyt,
            maxbibnames=9]{biblatex}
\DeclareFieldFormat{titlecase}{#1}
\crefname{equation}{Equation}{Equations}
\crefname{figure}{Figure}{Figures}
\crefname{lem}{Lemma}{Lemmas}
\crefname{cor}{Corollary}{Corollaries}
\crefname{prop}{Proposition}{Propositions}
\crefname{thm}{Theorem}{Theorems}
\crefname{mydef}{Definition}{Definitions}
\crefname{rem}{Remark}{Remarks}
\crefname{section}{Section}{Sections}

\counterwithin{equation}{section}

\newcommand{\xx}{\mathbf{x}}

\newcommand{\dd}{\mathrm{d}}
\newcommand{\Pe}{\mathrm{Pe}}

\newcommand{\discgrad}{D}
\newcommand{\dischess}{H}
\newtheorem{prop}{Proposition}
\newtheorem{lem}{Lemma}
\newtheorem{cor}{Corollary}
\newtheorem{thm}{Theorem}
\theoremstyle{definition}
\newtheorem{mydef}{Definition}
\newtheorem{rem}{Remark}

\numberwithin{prop}{section}
\numberwithin{lem}{section}
\numberwithin{cor}{section}
\numberwithin{thm}{section}
\numberwithin{mydef}{section}
\numberwithin{rem}{section}
\graphicspath{ {images/} }

\title{Convergence of a finite volume scheme for a model for ants}
\author{Maria Bruna, Markus Schmidtchen, Oscar de Wit}

\begin{document}
\maketitle
\begin{abstract}
    We develop and analyse a finite volume scheme for a nonlocal active matter system known to exhibit a rich array of complex behaviours. The model under investigation was derived from a stochastic system of interacting particles describing a foraging ant colony coupled to pheromone dynamics. In this work, we prove that the unique numerical solution converges to the unique weak solution as the mesh size and the time step go to zero. We also show discrete long-time estimates, which prove that certain norms are preserved for all times, uniformly in the mesh size and time step. In particular, we prove higher regularity estimates which provide an analogue of continuum parabolic higher regularity estimates. Finally, we numerically study the rate of convergence of the scheme, and we provide examples of the existence of multiple metastable steady states.
\end{abstract}

\section*{Mathematics Subject Classification}
35K55, 35B36, 65M08, 65M12, 35Q92, 92D50.

\section{Model introduction and motivation}

Active matter exhibits richly complex collective motion, ranging from flocking birds and schooling fish \cite{hemelrijk2012schools} to pattern-forming skin cells \cite{painter1999stripe}, turbulence in suspensions of microswimmers \cite{alert2022active}, and lane-forming pedestrians \cite{bacik2023lane}.
There are two main approaches to modelling collective phenomena in active matter: microscopic or agent-based models or macroscopic or PDE-based models \cite{BDT2017,BDT2019,BCT2022}. The former approach is very high-dimensional due to the large number of individuals typically involved in active matter systems, and quickly becomes analytically and computationally intractable. This is why macroscopic PDE models are customarily favoured to investigate the system's behaviour. Moreover, analytical challenges--commonly rooted in metastability and multiple bifurcation branches--highlight the importance of developing reliable numerical methods for macroscopic active matter models.
 
\begin{subequations}
\label{full_problem_intro}
In this paper, we develop and analyse a finite volume scheme for a nonlinear Fokker--Planck equation derived in \cite{bruna2024lane} as the continuum limit of a microscopic active matter model for ants:
 \begin{align}\label{eq:f}
 	 \partial_t f & =\nabla_\xx\cdot \left( D_T\nabla_\xx f-\Pe\mathbf{e}_\theta f \right) +\partial_\theta \left( \partial_\theta f-\gamma B[c] f \right), \quad f(t=0,\xx,\theta)=f^0(\xx,\theta),
 \end{align}
where $f = f(t, \xx, \theta): \mathbb R_+ \times \mathbb{T}_1^2 \times \mathbb T_{2\pi}\to \mathbb R_+$ is a one-particle probability density function that quantifies the probability to find a particle at time $t$ with position $\xx \in \mathbb{T}_1^2$ and orientation $\theta \in \mathbb T_{2\pi}$, and $B[c]$ is some functional that represents a nonlocal interaction between particles through a chemical field $c = c(t, \xx)$. In \cref{eq:f}, $D_T>0$ is the translational diffusion coefficient, $\Pe>0$ is the P\'eclet number (or dimensionless self-propulsion speed), $\mathbf{e}_\theta = (\cos\theta, \sin\theta)^\mathsf{T}$ is the direction of self-propulsion, and $\gamma>0$ is the interaction strength. 

Equation \eqref{eq:f} was derived in \cite{bruna2024lane} as the formal mean-field limit of a colony of ants interacting through the laying and sensing of pheromones. In particular, the pheromone field is coupled to the particle distribution through the elliptic equation
\begin{equation} \label{elliptic_c}
	\Delta c-\alpha c+\rho = 0,
\end{equation}
where $\alpha$ is a chemical decay rate and particles' spatial probability density is given by
\begin{equation}
	\rho(t,\xx) = \int_{\mathbb T_{2\pi}} f(t,\xx, \theta) \dd \theta.
\end{equation}
Ants interact nonlocally via the pheromone field according to 
\begin{equation}
    B_\lambda[c]=\mathbf{n}_\theta\cdot\nabla_\xx c(\xx+\lambda\mathbf{e}_\theta),\label{eq:Blambda}	
\end{equation}
where $\mathbf{n}_\theta=(-\sin\theta,\cos\theta)^\mathsf{T}$. \cref{eq:Blambda} models how ants adjust their orientation towards higher chemical concentrations by sensing at position $\xx+\lambda\mathbf{e}_\theta$, where $\lambda \ge 0$ represents the sensing distance of ants. This can be thought of as the location of antennas relative to their body centre. 
\end{subequations}

The solutions of model \eqref{full_problem_intro} were studied via a linear stability analysis and time-dependent numerical simulations in  \cite{bruna2024lane}. Specifically, for $\lambda>0$, \cref{full_problem_intro} admits three types of steady states depending on the model parameters: the trivial (constant) steady state, cluster stationary states, and lane-like steady states. In particular, in the region in the $(\gamma,\Pe)$-plane where the constant state is linearly unstable, numerical solutions of \cref{full_problem_intro} were found to either converge to stationary clusters or bidirectional lanes, which a thin region where bistability between the two types of nontrivial steady states is possible. 

Similar lane-like steady states were also found in \cite{bertucci2024curvature}, where they used a related interaction term $B[c]$ (which we discuss further below).

Model \eqref{full_problem_intro} is a nonlinear, nonlocal extension of the Fokker--Planck equation associated with a single active Brownian particle [corresponding to setting $\gamma = 0$ in \cref{eq:f}]. The self-propulsion term, premultiplied by the P\'eclet number $\Pe$, disrupts the gradient flow structure $\partial_t f = \nabla\cdot(f\nabla\delta\mathcal{F}/\delta f)$, often encountered in macroscopic PDE models for interacting particles.
The non-gradient character of \cref{eq:f} complicates the analysis as tools from gradient flows, such as a general LaSalle invariance principle \cite{carrillo2023invariance}, are not available. Despite these challenges, the well-posedness solutions $f\in L^2(0,T;H^1)$ for any time $T>0$ was proven in \cite{bruna2024lane}. The estimates rely on $D_T>0$ and classical Sobolev--Poincar\'e and Gagliardo--Nirenberg inequalities. Particularly, regularity estimates for the spatial density $\rho$, which satisfies 
\begin{equation}\label{eq:rho}
\partial_t\rho=\nabla\cdot(D_T\nabla \rho-\Pe\mathbf{p}),
\end{equation}
where $\mathbf{p}(t,\xx)=\int_0^{2\pi}\mathbf{e}_\theta f(t,\xx,\theta)\dd\theta$ is the polarization (or average orientation),
can be used to show regularity estimates for the full density $f$. The analysis in \cite{bruna2024lane} also gives long-time estimates showing that there cannot be blow up at infinite time, such that $f\in L^\infty(0,+\infty;L^\infty)$. A remaining challenge is the characterisation of the pattern-forming steady states beyond a stability analysis of the homogeneous state. 

Model \eqref{full_problem_intro} is a compelling example of pattern formation in a nonlinear active matter PDE model. Mathematical tools to characterise the existence of pattern formation were established for collectives of pedestrians in \cite{burger2016lane,bruna2022phase,burger2023well}, flocking birds \cite{degond2013macroscopic,briant2022cauchy}, collectives of migrating cells \cite{carrillo2018zoology} and suspensions of micro-swimmers \cite{ohm2022weakly,albritton2023stabilizing}. Despite the absence of a gradient flow structure, these papers provide the mathematical foundations for these pattern formations through bifurcation analysis, numerical methods, well-posedness analysis, and (non-)linear stability analysis. 

Our primary objective is to develop a numerical scheme that accurately reproduces the behaviour of \eqref{full_problem_intro}. Finite-volume schemes are particularly well-suited for this purpose, as their structure inherently ensures mass conservation and preserves the nonnegativity of the initial data. 
Extensive research exists on the numerical analysis of finite-volume methods \cite{eymard2000finite,filbet2006finite,zhou2017finite,chainais2007asymptotic,bailo2020convergence}, providing rigorous foundations for their convergence and error analysis across various PDE models. 
In particular, the works \cite{filbet2006finite,bailo2020convergence} are central to our analysis. 
Ref. \cite{filbet2006finite} demonstrates the convergence of a finite-volume scheme for the parabolic-elliptic Keller--Segel equation, which incorporates \eqref{elliptic_c}, using a compactness argument in $L^2(0,T;H^1)$.
A similar convergence result is given in \cite{bailo2020convergence} for a more general one-dimensional nonlinear aggregation-diffusion equation. Their analysis of general drift terms provides inspiration for treating the drift terms in our model, following a priori estimates for the chemical field in the discrete setting. These estimates can be obtained by translating the local well-posedness estimates in $L^2(0,T;H^1)$ of \cite{bruna2024lane} to the discrete setting, using discrete analogues of Sobolev--Poincar\'e inequalities \cite{bessemoulin2015discrete}.
In addition to a convergence result, we show the analysis for long-time estimates for the numerical solutions. As part of this, we show that a standard $H^1$-regularity parabolic estimate has a finite-volume equivalent, allowing us to show $L^2$- and $L^\infty$- long-time estimates. A crucial component of this analysis is a recent result on the Morrey inequality within the finite volume framework \cite[Theorem 4.1]{porretta2020note}.

The paper is organised as follows. In \cref{sec:definition}, we define the weak solutions of \eqref{full_problem_intro}, introduce the numerical scheme and state the main convergence result. In \cref{sec:fintest}, we provide the required a priori estimates for the convergence result, which is proven in \cref{sec:conv}. 
\cref{sec:highreg} covers the long-time estimates for the scheme. Finally, we present a series of numerical results in \cref{sec:numerics} showcasing the performance of our scheme. We show the scheme's order of accuracy and numerically verify the long-time behaviour observed in \cite{bruna2024lane,bertucci2024curvature}, confirming the convergence to nontrivial steady states. We conclude by showing that quasi-steady states comprising multiple spots or lanes exist and that they eventually destabilise and evolve to a single spot or lane.

\section{Definition of the scheme and function spaces} \label{sec:definition}
We begin by introducing some notation. Throughout, let $\xi = (\xx, \theta)$ denote the three-dimensional coordinate vector, $\Omega: = \mathbb T_1^2$ be a two-dimensional domain with periodic boundaries extending from $[-\frac{1}{2},\frac{1}{2}]$, $\Sigma= \Omega\times\mathbb T_{2 \pi}$ and $\Sigma_T = \Sigma \times (0, T)$. We will consider the nonlocal interaction \eqref{eq:Blambda}, as well as its zeroth and first-order Taylor expansions in $\lambda$, which we denote by $B_0$ and $B_\tau$ respectively:
\begin{subequations}
	\label{B_expanded}
	\begin{align} 
    B_0[c]&=\mathbf{n}_\theta\cdot\nabla_\xx c(\xx), \label{eq:B0}\\
    B_\tau[c]&=\mathbf{n}_\theta\cdot\nabla_\xx c(\xx)+\tau \mathbf{n}_\theta\cdot\nabla^2_\xx c(\xx)\mathbf{e}_\theta.\label{eq:Btau}
\end{align}
\end{subequations}
The sensing strategy $B_0$ corresponds to ants ``without antennas" sensing at their centre of mass. The $B_\tau$ term can be derived as the limit of a sensing mechanism for two antennas spread apart by some angle \cite{perna2012individual,bertucci2024curvature}. Typically, the terms $B_\lambda$ and $B_\tau$ lead to either aggregation or lane formation, whereas the term $B_0$ only leads to aggregation spots. Let us recall our notion of weak solution of \cref{full_problem_intro}.
\begin{mydef}[Weak solution]\label{def:solf}
    A function $f\in L^2(0,T;H^1(\Sigma))$ is a weak solution for \cref{full_problem_intro} if $f(0)=f^0 \in L_+^1(\Sigma)\cap L^\infty_+(\Sigma)$ and
    \begin{equation}\label{eq:weakeq}
        \int_0^T\int_\Sigma \left[f\partial_t\varphi -(D_T\nabla_\xx f-\Pe\mathbf{e}_\theta f)\cdot\nabla_\xx\varphi-(\partial_\theta f-\gamma B[c] f)\partial_\theta\varphi \right]\dd\xi\dd t + \int_\Sigma f(0)\varphi(0)\dd\xi = 0,
    \end{equation}
    for all $\varphi\in C^\infty(\overline{\Sigma_T})$ such that $\varphi(T)=0$ and $c \in L^2(0,T; H^2(\Omega))$ is the unique strong solution of \cref{elliptic_c}.
    \end{mydef}

Next, we introduce a discretisation of the domain $\Sigma$ and a time interval $(0, T)$.
\begin{mydef}[Discretisation of the domain]
For some $N_x, N_y, N_\theta \in \mathbb N$, we introduce a family of cells
      \begin{equation}
        C_{i,j,k}=[x_{i-1/2},x_{i+1/2})\times[y_{j-1/2},y_{j+1/2})\times[\theta_{k-1/2},\theta_{k+1/2}),
    \end{equation}
    for
    \begin{align}
        (i,j,k) \in \mathcal{I} = \{(i,j,k) \,|\,  1\leq i\leq N_x,1\leq j\leq N_y,1\leq k\leq N_\theta\},
    \end{align}
 where
 \begin{align*}
    x_{i-1/2} = -\frac{1}{2}+(i-1)\Delta x, \quad y_{j-1/2}= -\frac{1}{2}+(j-1)\Delta y, \quad \text{ and } \quad \theta_{k-1/2}=(k-1)\Delta \theta,    
 \end{align*}
 with $\Delta x=1/N_x,\Delta y=1/N_y$ and $\Delta\theta=2\pi/N_\theta$. We note that $\bigcup_{\mathcal I} \overline{C_{i,j,k}}=\Sigma$.
    We also define the dual mesh cells as 
    \begin{equation}
        C_{i+1/2,j,k}=[x_{i},x_{i+1})\times[y_{j-1/2},y_{j+1/2})\times[\theta_{k-1/2},\theta_{k+1/2}),
    \end{equation}
  with $x_i=-\frac{1}{2}+(i-\frac{1}{2})\Delta x$ for $1 \le i \le N_x$ (and $x_{N_x + 1} \equiv x_1$), and analogously for the $y$- and $\theta$-direction. We use the notation $\xx_{i,j}=(x_i,y_j)$.
  
For the time interval, we introduce the subintervals $[t^n, t^{n+1})$, where $t^n=n\Delta t$, such that 
\begin{equation}
    \bigcup_{n=0,1,\dots,N_T-1}[t^n ,t^{n+1})=[0,T),
\end{equation}
and $N_T=T/\Delta t$ is a positive integer.
\end{mydef}

We are now ready to introduce the numerical scheme using the above notation.
\begin{mydef}[Definition of the scheme]\label{def:scheme}
    We construct a discretised version of the initial data $f^0$ via the cell averages
\begin{equation}
    f^0_{i,j,k} = \frac{1}{\Delta \xi}\int_{C_{i,j,k}}f^0(\xi)\dd \xi,
\end{equation}
for $(i,j,k) \in \mathcal{I}$, where $\Delta\xi =\Delta\xx\Delta\theta$, and $\Delta\xx=\Delta x\Delta y$. To discretise \cref{eq:f}, we rewrite it in divergence form $\partial_t f=-\nabla_\xi\cdot \mathbf{F}$, with flux $\mathbf{F}=(F^x,F^y,F^\theta)$ 
\begin{align}
\begin{aligned}
	F^x(t,\xi) &=-\left (D_T\partial_x f-\Pe\cos(\theta)f \right),\\
    F^y(t,\xi) & =- \left( D_T\partial_y f-\Pe\sin(\theta)f \right),\\
    F^\theta(t,\xi) & =- \left (\partial_\theta f-\gamma B[c]f \right).
\end{aligned}
\end{align}
Integrating \cref{eq:f} over a test cell, $[t^{n},t^{n+1})\times C_{i,j,k}$, yields
\begin{equation}
\begin{split}\label{eq:fflux}
    \int\limits_{C_{i,j,k}}f(t^{n+1},\xi)\dd\xi-\int\limits_{C_{i,j,k}}f(t^n,\xi)\dd\xi=&-\int\limits_{t^n}^{t^{n+1}}\int\limits_{C_{i,j,k}}\partial_x F^x(t,\xi)\dd\xi\dd t-\int\limits_{t^n}^{t^{n+1}}\int\limits_{C_{i,j,k}}\partial_y F^y(t,\xi)\dd\xi\dd t\\&-\int\limits_{t^n}^{t^{n+1}}\int\limits_{C_{i,j,k}}\partial_\theta F^\theta(t,\xi)\dd\xi\dd t,
\end{split}
\end{equation}
Using the fundamental theorem turning the derivatives on the right-hand side into finite differences, we can approximate Equation \eqref{eq:fflux} by the backward finite volume scheme
\begin{subequations}
    \label{eq:scheme}
\begin{equation} \label{eq:schemef}
    \begin{split}
        \frac{f^{n+1}_{i,j,k}-f^n_{i,j,k}}{\Delta t}=&-\frac{1}{\Delta x}\left(F^{x,n+1}_{i+1/2,j,k}-F^{x,n+1}_{i-1/2,j,k}\right)-\frac{1}{\Delta y}\left(F^{y,n+1}_{i,j+1/2,k}-F^{y,n+1}_{i,j-1/2,k}\right)\\
        &-\frac{1}{\Delta\theta}\left(F^{\theta,n+1}_{i,j,k+1/2}-F^{\theta,n+1}_{i,j,k-1/2}\right),
    \end{split}
\end{equation}
where the discrete fluxes are defined as 
\begin{align}
\begin{aligned}
	    F^{x,n}_{i+1/2,j,k}&=-\left(D_T\dd_x f^n_{i+1/2,j,k}-\Pe U^n_{i+1/2,j,k}\right),\\
    F^{y,n}_{i,j+1/2,k}&=-\left(D_T\dd_y f^n_{i,j+1/2,k}-\Pe V^n_{i,j+1/2,k}\right),\\
    F^{\theta,n}_{i,j,k+1/2}&=-\left(\dd_\theta f^n_{i,j,k+1/2}-\gamma W^n_{i,j,k+1/2}\right),
    \end{aligned}
\end{align}
and discrete partial derivatives defined as
\begin{align}
    \label{eq:disclap}
    \begin{split}
    \dd_x f_{i+1/2,j,k} &= \dfrac{1}{\Delta x}(f_{i+1,j,k}-f_{i,j,k}),\\
    \dd_y f_{i,j+1/2,k} &= \dfrac{1}{\Delta y}(f_{i,j+1,k} - f_{i,j,k}),\\
    \dd_\theta f_{i,j,k+1/2} &= \dfrac{1}{\Delta \theta}(f_{i,j,k+1}-f_{i,j,k}).\\
    \end{split}
\end{align}
The upwind velocities for the drift terms are, using the notation $(\cdot)^+=\max(\cdot,0),(\cdot)^-=\min(\cdot,0)$,
\begin{align} \label{eq:scheme_fluxes}
\begin{aligned}
    U^{n}_{i+1/2,j,k} & =(\cos \theta_k)^+ f^n_{i,j,k}+(\cos \theta_k )^- f^n_{i+1,j,k},\\
    V^{n}_{i,j+1/2,k} &=(\sin \theta_k )^+ f^n_{i,j,k}+(\sin \theta_k )^- f^n_{i,j+1,k},\\
    W^{n}_{i,j,k+1/2} &=(B[c^n]_{i,j,k+1/2})^+ f_{i,j,k}+(B[c^n]_{i,j,k+1/2})^- f_{i,j,k+1}.
    \end{aligned}
\end{align}
Next, we discretise \cref{elliptic_c} using finite differences  
\begin{equation}
    \label{eq:discc}
    0=\left(\dd_x^2+\dd_y^2\right)c_{i,j}^n-\alpha c_{i,j}^n+\rho_{i,j}^n,
\end{equation}
where the Laplacian terms are defined as 
\begin{equation}
    \dd_x^2 c^n_{i,j}=\frac{c^n_{i+1,j} - 2c^n_{i,j} + c^n_{i-1,j}}{\Delta x^2},
\end{equation}
and analogously in the $y$-direction. Here, $\rho_{i,j}^n=\sum_k f_{i,j,k}^n\Delta\theta$ denotes the discretised spatial density.

Lastly, we use the following discretisations for the three different modelling choices for the interaction term $B[c]$. 
\begin{itemize}
    \item 
For $B_0[c]$, we define 
\begin{equation} \label{eq:B0_sch}
    B_0[c]_{i,j,k+1/2}=\mathbf{n}(\theta_{k+1/2})\cdot\discgrad c_{i,j},
\end{equation} 
and the gradient is discretised as a centred finite difference
\begin{equation}\label{eq:cgrad}
    \discgrad c_{i,j}=  \frac{1}{2}\begin{pmatrix}
        \dfrac{1}{\Delta x}(c_{i+1,j}-c_{i-1,j})\\[1em]
        \dfrac{1}{\Delta y}(c_{i,j+1}-c_{i,j-1})
    \end{pmatrix}=\begin{pmatrix}
        \discgrad_x c_{i,j}\\ \discgrad_y c_{i,j}
    \end{pmatrix}=\frac{1}{2}\begin{pmatrix}
        \dd_x c_{i+1/2,j}+\dd_x c_{i-1/2,j}\\
        \dd_y c_{i,j+1/2}+\dd_y c_{i,j-1/2}
    \end{pmatrix}.
\end{equation} 

\item For $B_\lambda[c]$, we define
\begin{equation}\label{eq:blambdadisc}
    B_\lambda[c]_{i,j,k+1/2}=\mathbf{n}(\theta_{k+1/2})\cdot\discgrad c_{i,j,k+1/2},
\end{equation}
where $c_{i,j,k+1/2}$ approximates $c(\xx_{i,j}+\lambda\mathbf{e}(\theta_{k+1/2}))$ and it is defined using nearest-neighbour interpolation to access $c$ at the points $\xx_{i,j}+\lambda\mathbf{e}(\theta)$, see \cref{def:discblambda} for a full definition. 

\item For $B_\tau[c]$, we define \begin{equation}\label{eq:btaudisc}
    B_\tau[c]_{i,j,k+1/2}=\mathbf{n}(\theta_{k+1/2})\cdot\discgrad c_{i,j}+\tau\mathbf{n}(\theta_{k+1/2})\cdot\dischess c_{i,j}\mathbf{e}(\theta_{k+1/2}).
\end{equation}
where the discrete Hessian is defined as
\begin{equation}
    \dischess c_{i,j} =
    \begin{pmatrix}
        \dd_x^2 c_{i,j} & \discgrad_x \discgrad_y c_{i,j}\\
        \discgrad_x \discgrad_y c_{i,j} & \dd_y^2 c_{i,j}
    \end{pmatrix},
\end{equation}
where the off-diagonal terms read
$$\discgrad_x \discgrad_y c_{i,j} =  \frac{1}{4\Delta\xx} \left(  c_{i+1,j+1}-c_{i-1,j+1}-c_{i+1,j-1}+c_{i-1,j-1}\right).
$$
\end{itemize}
\end{subequations} 
\end{mydef}
\begin{mydef}[Piecewise constant interpolations]
    \label{def:pwc-soln}
    Given $(f_{i,j,k}^n)_{(i,j,k,n)}$ a discrete solution to the scheme, we define the piecewise constant interpolation as 
    \begin{equation}
        f_h(t,x,y,\theta)= f^n_{i,j,k}, \quad  \mathrm{for} \quad (t,x,y,\theta)\in [t^{n-1},t^{n})\times C_{i,j,k},
    \end{equation}
    where the subscript $h$ corresponds to the mesh size
    \begin{equation}
        h=\max\{\Delta t,\Delta x,\Delta y,\Delta\theta\}.
    \end{equation}
    We also use straightforward analogous definitions of piecewise constant interpolations for functions with only $(t,x,y)$ as variables.

We define the discrete gradients on the dual mesh as
    \begin{equation}
        \dd_x f_h(t,x,y,\theta) = \dd_x f^n_{i+1/2,j,k}\quad  \mathrm{for} \quad (t,x,y,\theta)\in [t^{n-1},t^n) \times C_{i+1/2,j,k},
    \end{equation}
    and analogously for the gradients in $y$ and $\theta$.
\end{mydef}
We show the following well-posedness result for finite-volume scheme of \cref{def:scheme}. We postpone its proof to the next section.
\begin{prop}[Existence, uniqueness and nonnegativity]\label{prop:exist}
    Let $f^0\in L^1_+(\Sigma)\cap L^\infty_+(\Sigma)$ be a nonnegative initial datum with mass one. Then, for any $N_x, N_y, N_\theta, N_T \in \mathbb N$ there is a unique nonnegative solution $(f^n_{i,j,k})$ to the scheme of \cref{eq:scheme} for $(n,i,j,k)\in\{0,1,\dots,N_T\}\times\mathcal{I}$. Moreover, the mass is conserved and there holds $\|f_h(t)\|_{L^1}=\|f^0\|_{L^1}$, for all $t\in[0,T)$.
\end{prop}
We conclude this section by stating the main convergence result, which we will prove later.
\begin{thm}[Convergence of the scheme]\label{thm:conv}
    Let $f^0\in L^1_+(\Sigma)\cap L^\infty_+(\Sigma)$ be a nonnegative initial datum with mass one. Then, given a family of solutions $(f_h)_h$ as defined in Definition \ref{def:pwc-soln} and provided $\Delta t<\max\{\frac{D_T}{2\Pe^2},C'C_T\}$ (see \cref{cor:unifL2}), there exists a subsequence that converges strongly in $L^2(\Sigma_T)$, as $h\to 0$, and its limit is a weak solution of \cref{full_problem_intro} as in \cref{def:solf}.
\end{thm}


\section{A Priori Estimates}
\label{sec:fintest}
In this section, we prove the existence result of \cref{prop:exist} and derive a set of discrete a priori estimates necessary for \cref{thm:conv}. To this end, we recall a discrete Gr\"onwall-type inequality mimicking the standard continuum Gr\"onwall inequality. 
Next, we define discrete analogues of Sobolev space norms. We then prove an $L^2$-estimate for the discrete spatial density $\rho_h$, an $H^1$-type estimate for the discrete chemical field $c_h$ and the discrete interaction term $\dd_\theta B[c_h]$. Finally, we obtain a uniform-in-time $L^2$-estimate for $f_h$ and $\dd_\xi f_h$ on bounded time intervals, independent of the mesh size. All these estimates will be used to prove the convergence of the scheme in \cref{sec:conv}.

\begin{lem}[Discrete Gr\"onwall inequality]\label{lem:discgron}
    Suppose $(F^n)_n \subset (0, \infty)$ is a nonnegative real-valued sequence and $C\in[0,1)$. If
    \begin{equation}
        F^{n+1}-F^n\leq C F^{n+1}, \quad \mathrm{for} \quad n=0,1,2,\dots,
    \end{equation}
    then 
    \begin{equation}
        F^n\leq F^0\prod_{k=1}^n\frac{1}{1-C^k}.    
    \end{equation}
\end{lem}
Next, let us define the discrete Sobolev spaces we will use in our analysis.
\begin{mydef}[Function space norms]
For $1\leq p<+\infty$, we define the $L^p$-norm
\begin{equation}
        \|f_h\|_{L^p}=\left(\int_\Sigma |f_h(t,\xi)|^p\dd\xi\right)^{1/p}=\left(\sum_{i,j,k}|f^n_{i,j,k}|^p\Delta\xi\right)^{1/p}.
\end{equation}
We define the $L^\infty$-norm as 
\begin{equation}
    \|f_h\|_{L^\infty}=\sup_{\xi\in\Sigma}|f_h(t,\xi)|=\sup_{i,j,k}|f^n_{i,j,k}|.
\end{equation}
For $1\leq p<+\infty$ we define the Sobolev seminorm
\begin{align}
    |f_h|_{1,p}
    &=\left(\sum_{i,j,k}\left(|\dd_x f_{i+1/2,j,k}|^p+|\dd_y f_{i,j+1/2,k}|^p+|\dd_\theta f_{i,j,k+1/2}|^p\right)\Delta\xi\right)^{1/p}\\
    &=\|\dd_\xi f_h\|_{L^p},
\end{align}
where
\begin{equation}
    \dd_\xi f_h(t,x,y,\theta)=(\dd_x f_h,\dd_y f_h,\dd_\theta f_h)^\mathsf{T}.
\end{equation}
The discrete Sobolev norm is defined as 
\begin{equation}
    \|f_h\|_{1,p}=|f_h|_{1,p}+\|f_h\|_{L^p}.
\end{equation}
We also use the notation $\|f_h\|_{1,p,\Sigma}$ if it is relevant what domain, here $\Sigma$, is used for the integration or summation.
\end{mydef}
\begin{lem}[Summation by parts]
\label{lem:ibp}
For $(a_{i,j,k})_{(i,j,k)\in \mathcal I}$ and $(b_{i,j,k})_{(i,j,k)\in \mathcal I}$ on the discretised periodic domain we have
    \begin{equation}
        \sum_{i,j,k}(a_{i+1,j,k}-a_{i,j,k})b_{i,j,k}=-\sum_{i,j,k}(b_{i+1,j,k}-b_{i,j,k})a_{i+1,j,k}.
    \end{equation}
Analogous relations hold in the $j$- and $k$-directions.
\end{lem}
\begin{proof}
We write
\begin{align*}
    \sum_{i,j,k}(a_{i+1,j,k}-a_{i,j,k})b_{i,j,k}=&\sum_{i,j,k}a_{i+1,j,k}b_{i,j,k}-\sum_{i,j,k}a_{i,j,k}b_{i,j,k}\\ =& \sum_{i,j,k}a_{i+1,j,k}b_{i,j,k}-\sum_{i,j,k}a_{i+1,j,k}b_{i+1,j,k},
\end{align*}
after relabelling and using periodicity, that is, $a_{N_x+1,j,k}=a_{1,j,k}$.
\end{proof}
We can now prove \cref{prop:exist}.
\begin{proof}[Proof of \cref{prop:exist}]
We first establish the existence and uniqueness of the finite difference \cref{eq:discc}. For a given nonnegative $(\rho^n_{i,j})_{i,j}$, we note that we have a linear system of $N_x  N_y$ unknowns $c_{i,j}^n$ and the same number of coupled equations
     $$
     -\dd^2_x c^n_{i,j} -\dd^2_y c^n_{i,j} + \alpha c^n_{i,j}=\rho^n_{i,j}.
     $$
We next show that the kernel of this linear system is the trivial solution $c^n_{i,j}= 0$. 
Indeed, for $\rho^n_{i,j}=0$, multiplying \cref{eq:discc} by $c^n_{i,j}$ and summing over all pairs $(i,j)$, summation by parts yields 
     \begin{equation}
         \sum_{i,j}\left( |\dd_x c_{i,j}^n|^2
          + |\dd_y c_{i,j}^n|^2 + \alpha |c^n_{i,j}|^2\right)=0.
     \end{equation}
Hence, we find $c^n_{i,j}=0$. This implies that for any $\rho^n_{i,j}$, there is a unique solution to the linear system \cref{eq:discc}. Further, since $\rho^n_{i,j}$ is nonnegative by assumption, $c^n_{i,j}$ is nonnegative by the same argument as in the proof of \cite[Theorem 3.1]{bailo2020convergence}.

Now, we construct the fixed-point operator. First, for a given prescribed chemical field $c^n_{i,j}$, possibly varying in time, and an initial datum $f^0_{i,j,k}$ satisfying the assumptions, there exists a unique solution $f^n_{i,j,k}$ to the scheme \eqref{eq:scheme} (without \cref{eq:discc}), see \cite[Theorem 17.1]{eymard2000finite}. This solution is also nonnegative and mass-preserving by the same arguments as in the proof of \cite[Theorem 3.1]{bailo2020convergence}. Upon setting
     $$
        \mathcal X = \left\{ (c_{i,j}^n)\in\mathbb{R}^{N_T N_x N_y} \ \Bigg| \ \sup_{n=0,1,\dots,N_T}\sum_{i,j}|c_{i,j}^n|\Delta\xx\leq 1/\alpha\right\},
     $$
    which is convex and compact, we define the operator 
     $$
     \mathcal{S}: \mathcal X \to \mathcal X, c^n_{i,j}\mapsto f^n_{i,j,k}\mapsto \tilde c^n_{i,j},
     $$
     where $\tilde c^n_{i,j}$ solves \cref{eq:discc} with $\rho^n_{i,j}=\sum_k f^n_{i,j,k}\Delta\theta$ as its source term. 
     This operator is continuous (cf. proof of \cite[Theorem 2.1]{filbet2006finite}). Also, we have that $\sum_{i,j} \tilde c^n_{i,j}\Delta\xx=1/\alpha$ for all $n=0,1,\dots,N_T$, using the mass-preserving property and \cref{eq:discc}. Hence, by the nonnegativity of $\tilde c^n_{i,j}$, the operator $\mathcal{S}$ is a fixed-point operator. Applying Brouwer's fixed point theorem, this concludes the proof.
\end{proof}

\begin{rem}[Scheme for $\rho_{i,j}^n$] 
    Multiplying \cref{eq:schemef} by $\Delta\theta$ and summing over $k$, results in the following scheme for the spatial density, 
    \begin{subequations}
    	\label{eq:schemerho}
    \begin{align}
        \label{eq:scheme-rho}
        \frac{\rho^{n+1}_{i,j}-\rho^n_{i,j}}{\Delta t}=-\frac{1}{\Delta x}\left(\bar{F}^{x,n+1}_{i+1/2,j}-\bar{F}^{x,n+1}_{i-1/2,j}\right)-\frac{1}{\Delta y}\left(\bar{F}^{y,n+1}_{i,j+1/2}-\bar{F}^{y,n+1}_{i,j-1/2}\right),
    \end{align}
    with
    \begin{equation}
        \label{eq:fluxes-rho}
        \begin{split}
            &\bar{F}^{x,n+1}_{i+1/2,j}= - \left(D_T \dd_x\rho^n_{i+1/2,j}-\Pe\bar{U}^{n+1}_{i+1/2,j}\right), \quad \text{where} \quad \bar{U}^{n+1}_{i+1/2,j}=\sum_{k} U^{n+1}_{i+1/2,j,k}\Delta\theta,\\
            &\bar{F}^{y,n+1}_{i,j+1/2}= - \left(D_T \dd_y\rho^n_{i,j+1/2}-\Pe\bar{V}^{n+1}_{i,j+1/2}\right), \quad \text{where} \quad \bar{V}^{n+1}_{i,j+1/2}=\sum_{k}V^{n+1}_{i,j+1/2,k}\Delta\theta,
        \end{split}
    \end{equation}
        \end{subequations}
    where $U^{n}_{i+1/2,j,k}, V^{n}_{i+1/2,j,k}$ are given in \cref{eq:scheme_fluxes}.
As expected, it corresponds to the finite-volume discretisation of \cref{eq:rho}. We define the piecewise constant interpolation $\rho_h(t,\xx)$ from $\rho_{i,j}^n$ analogously to \cref{def:pwc-soln}.
\end{rem}
\begin{prop}[$L^2$-estimate for $\rho_h$]\label{prop:L2rho}
    Provided that $\Delta t<\frac{D_T}{2\Pe^2}$ and $\rho^0\in L^1_+(\Omega)\cap L^\infty_+(\Omega)$, then the spatial density of the solution for the scheme of \cref{eq:scheme} satisfies
    \begin{equation}
        \sup_{t\in[0,T]}\|\rho_h(t)\|_{L^2}\leq C_T,
    \end{equation}
    where $C_T>0$ is independent of $h>0$ but may depend on $T>0$ and $\|\rho^0\|_{L^2}$.
\end{prop}
\begin{proof}
We multiply \cref{eq:scheme-rho} by $\rho_{i,j}^{n+1}$ and sum over all $i,j$ to get
        \begin{equation} 
        \begin{split}
            \sum_{i,j} \frac{\rho^{n+1}_{i,j}-\rho^n_{i,j}}{\Delta t}\rho^{n+1}_{i,j}
            &=
            D_T\sum_{i,j} \left (\dd_x^2\rho^{n+1}_{i,j} + \dd_y^2\rho^{n+1}_{i,j} \right) \rho^{n+1}_{i,j}\\
        &\phantom{=}-\Pe\sum_{i,j}\left(\frac{\bar{U}^{n+1}_{i+1/2,j}-\bar{U}^{n+1}_{i-1/2,j}}{\Delta x} \rho^{n+1}_{i,j} + \frac{\bar{V}^{n+1}_{i,j+1/2}-\bar{V}^{n+1}_{i,j-1/2}}{\Delta y}\rho^{n+1}_{i,j}\right).
        \end{split}
        \end{equation}
Then, multiplying by $\Delta\xx$ and summing by parts as in \cref{lem:ibp}, 
    we get
    \begin{equation} \label{schrho}
    \begin{aligned}
       \sum_{i,j}\frac{\rho^{n+1}_{i,j}-\rho^n_{i,j}}{\Delta t} \rho^{n+1}_{i,j}\Delta\xx 
       &= - D_T\sum_{i,j}
        \left(|\dd_x\rho_{i+1/2,j}^{n+1}|^2 + |\dd_y\rho_{i,j+1/2}^{n+1}|^2\right)\Delta\xx  \\
        & \phantom{{} =} +  \Pe \sum_{i,j} \left(\bar{U}^{n+1}_{i+1/2,j}\dd_x\rho_{i+1/2,j}^{n+1}+\bar{V}^{n+1}_{i,j+1/2}\dd_y\rho_{i,j+1/2}^{n+1}\right)\Delta\xx\\
        &\leq -(D_T-\varepsilon)\|\dd_\xx\rho_h(t^{n+1})\|_{L^2(\Omega)}^2+\frac{\Pe^2}{\varepsilon}\|\rho_h(t^{n+1})\|_{L^2(\Omega)}^2,
    \end{aligned}
    \end{equation}
using $|\bar{U}^{n+1}_{i+1/2,j}|\leq\max\{\rho^{n+1}_{i,j},\rho^{n+1}_{i+1,j}\}$, the weighted Young's  inequality, and the notation
    \begin{equation}
        \dd_\xx\rho_h=(\dd_x\rho_h,\dd_y\rho_h)^\mathsf{T}.
    \end{equation}
     Finally, using the inequality $b(b-a)\geq \tfrac{1}{2}(b^2-a^2)$ on the left-hand side of \cref{schrho}, we obtain
    \begin{equation}
        \frac{1}{2\Delta t}\left(\|\rho_h(t^{n+1})\|_{L^2(\Omega)}^2-\|\rho_h(t^n)\|_{L^2(\Omega)}^2\right)\leq -(D_T-\varepsilon)\|\dd_\xx\rho_h(t^{n+1})\|_{L^2(\Omega)}^2+\frac{\Pe^2}{\varepsilon}\|\rho_h(t^{n+1})\|_{L^2(\Omega)}^2.
    \end{equation}
    Thus, 
     \begin{equation}
        \|\rho_h(t^{n+1})\|_{L^2(\Omega)}^2-\|\rho_h(t^n)\|_{L^2(\Omega)}^2 \leq\frac{2\Delta t\Pe^2}{D_T}\|\rho_h(t^{n+1})\|_{L^2(\Omega)}^2.
    \end{equation}
    Hence, by \cref{lem:discgron}, the inequality $1+x\leq \mathrm{e}^x$ for $x\in\mathbb{R}_{\geq 0}$, and provided that $\Delta t<\frac{D_T}{2\Pe^2}$, the result follows.
\end{proof}

\begin{rem}[Hessian of $c_h$]  For the chemical field $c$, we define the piecewise constant Hessian
\begin{equation}
    \dd_\xx^2 c_h=\begin{pmatrix}
        \dd_{x}^2 c_h & \dd_{xy} c_h\\
        \dd_{xy} c_h & \dd^2_y c_h
    \end{pmatrix},
\end{equation}
such that 
\begin{equation}
    \dd_x^2 c_h=\dd_x^2 c_{i,j} \quad \mathrm{for} \ (x,y)\in C_{i,j},  \qquad \dd_{xy}c_h=\dd_y \dd_{x} c_{i+1/2,j+1/2} \quad \mathrm{for} \ (x,y)\in C_{i+1/2,j+1/2}.
\end{equation}
\end{rem}

\begin{prop}\label{cor:disccreg}
There exist constants $C,C_B>0$ such that,  for any choice of $B\in\{B_0,B_\lambda,B_\tau\}$, the following inequalities hold for $\rho_h$ and $c_h$ numerical solutions of \cref{eq:discc}:
    \begin{align}
    \|c_h\|_{1,2}\leq& C\|\rho_h\|_{L^2(\Omega)}, \qquad \|\dd_\xx^2 c_h\|_{L^2}\leq C\|\rho_h\|_{L^2(\Omega)}
    \end{align}
    and
    \begin{align}
        \|\dd_\theta B[c_h]\|_{L^2} \leq & C_B\|\rho_h\|_{L^2(\Omega)}.
    \end{align}
The constants $C,C_B>0$ do not depend on $h$, as long as the ratios $C_{\Delta x}=\Delta\theta/\Delta x,C_{\Delta y}=\Delta\theta/\Delta y$ are  uniformly bounded from below strictly from zero, that is, $\inf_h \min\{C_{\Delta x}(h),C_{\Delta y}(h)\}\geq \varepsilon_C$ for some $\varepsilon_C>0$. 
\end{prop}
\begin{proof}
    First, we derive pointwise estimates on $\dd_\theta B[c_h]$ in terms of the discrete gradient and discrete Hessian of $c_h$. Then, we derive estimates for the $L^2$-norms of the discrete gradient and Hessian of $c_h$ in terms of the $L^2$-norm of $\rho_h$. 
    
    We note that for all the modelling choices for the interaction terms for $B[c]$ we have \begin{equation}
        \dd_\theta B[c]_{i,j,k}=\frac{1}{\Delta \theta}\left(B[c]_{i,j,k+1/2}-B[c]_{i,j,k-1/2}\right).
    \end{equation}
    \begin{itemize}
    	\item $B_0$ interaction: using the mean-value theorem, we have
    \begin{equation}
        |\dd_\theta B_0[c]_{i,j,k}|=\left|\frac{1}{\Delta\theta}[\mathbf{n}(\theta_{k+1/2})-\mathbf{n}(\theta_{k-1/2})]\cdot\discgrad c_{i,j}\right| \leq |\discgrad c_{i,j}|.
    \end{equation}
    \item $B_\tau$ interaction:  we have
    \begin{equation}
    \begin{split}
        \dd_\theta B_\tau[c]_{i,j,k}=&\frac{1}{\Delta\theta}[\mathbf{n}(\theta_{k+1/2})-\mathbf{n}(\theta_{k-1/2})]\cdot\discgrad c_{i,j}\\
        &+\frac{\tau}{\Delta\theta}[\mathbf{n}(\theta_{k+1/2})\dischess c_{i,j}\mathbf{e}(\theta_{k+1/2})-\mathbf{n}(\theta_{k-1/2})\dischess c_{i,j}\mathbf{e}(\theta_{k-1/2})].
    \end{split}
    \end{equation}
    Using the mean value theorem, we  obtain
    \begin{align}
        |\dd_\theta B_\tau[c]_{i,j,k}|  \leq C (|\discgrad c_{i,j}| + |\dischess c_{i,j}|),
    \end{align}
    for some constant $C>0$ independent of $h$. 
    \item $B_\lambda$ interaction: a pointwise estimate is given in \cref{sec:appendix}, and we prove that there is a constant $C>0$ independent of the mesh size such that 
    \begin{equation}
        \|\dd_\theta B[c_h]\|_{L^2}^2\leq C(\|\dd_\xx c_h\|_{L^2(\Omega)}^2+\|\dd^2_{\xx}c_h\|_{L^2(\Omega)}^2),
    \end{equation} for the details of the proof see \cref{sec:appendix}.
    \end{itemize}
    Now we derive the $L^2$-estimates on $c_h$ in terms of the $L^2$-norm of $\rho_h$. Multiplying \cref{eq:discc} by $c_{i,j}\Delta\xx$ and summing over all doubles $(i,j)$ and using integration by parts gives 
    \begin{equation}\label{eq:L2c}
        \|\dd_\xx c_h\|_{L^2(\Omega)}^2+\alpha\|c_h\|_{L^2(\Omega)}^2=\|\rho_h c_h\|_{L^1(\Omega)}\leq \frac{1}{2\alpha}\|\rho_h\|_{L^2(\Omega)}^2+\frac{\alpha}{2}\|c_h\|_{L^2(\Omega)}^2.
    \end{equation}
    This gives the inequality to be obtained, $\|\dd_\theta B[c_h]\|_{L^2}\leq C_B\|\rho_h\|_{L^2(\Omega)}$, for the $B_0$ interaction.
    For the Hessian appearing in the $B_\tau$ and $B_\lambda$ interaction terms, we want to derive a discrete equivalent of the $L^2$-elliptic regularity estimate, i.e., $\|\nabla^2 c\|_{L^2(\Omega)}=\|\Delta c\|_{L^2(\Omega)}$. 
    By integration by parts, we find
    \begin{equation} 
        \label{control_cross}
        \sum_{i,j}  \dd_x^2 c_{i,j}\dd_y^2 c_{i,j} \Delta\xx = \sum_{i,j} |\dd_{x}\dd_y c_{i+1/2,j+1/2}|^2\Delta\xx.
    \end{equation}
    Hence, the $L^2$-norm of the mixed derivatives is controlled by the diagonal terms of $\dd_\xx^2 c$, just as with the continuum PDE level. As a result, 
 squaring \cref{eq:discc} and using $ab\leq\tfrac{1}{2}(a^2+b^2)$ gives
    \begin{equation}
        (\dd_x^2 c_{i,j}+\dd_y^2c_{i,j})^2=\alpha^2 (c_{i,j})^2+2\alpha c_{i,j}\rho^n_{i,j}+(\rho^n_{i,j})^2\leq 2\alpha^2(c_{i,j})^2+2(\rho^n_{i,j})^2.
    \end{equation}
    Summing over doubles $(i,j)$ and multiplying by $\Delta\xx$ and using \cref{eq:L2c} to control $\|c_h\|_{L^2(\Omega)}$ by $\|\rho_h\|_{L^2(\Omega)}$, gives the result for $B_\tau$ and $B_\lambda$ as well.
\end{proof}
\begin{prop}[$L^2$-estimate for $f_h$]\label{cor:unifL2}
Provided that $\Delta t<\frac{D_T}{2\Pe^2}$, $\Delta t< C' C_T$, where $C'>0$ is given in the proof below and $C_T$ is as in \cref{prop:L2rho}, and $\|f^0\|_{L^2}<+\infty$, there exists a constant $C>0$ independent of the mesh size but depending on the final time $T>0$ such that the solutions to the numerical scheme \eqref{eq:scheme} satisfy
\begin{equation}
    \sup_{t\in[0,T]}\|f_h(t)\|_{L^2}+\int_{0}^T\|\dd_\xi f_h(t)\|_{L^2}^2\dd t\leq C.
\end{equation}
\end{prop}
\begin{proof}
    The idea for the proof is a discrete analogue of \cite[Lemma 3.3]{bruna2024lane}. We multiply \cref{eq:scheme} by $f_{i,j,k}^{n+1}\Delta\xi$ and sum over all triples $(i,j,k)$. This gives 
\begin{multline}
            \sum_{i,j,k}\frac{\Delta\xi}{\Delta t}\left(f^{n+1}_{i,j,k}-f^n_{i,j,k} \right) f^{n+1}_{i,j,k} =\sum_{i,j,k}\Delta\xi \left[D_T \left(\dd_x^2+\dd_y^2 \right) f^{n+1}_{i,j,k} f^{n+1}_{i,j,k}+\dd_\theta^2 f^{n+1}_{i,j,k} f^{n+1}_{i,j,k} \right]\\
        -\sum_{i,j,k}\Pe\Delta\xi\left[\frac{1}{\Delta x} \left( U^{n+1}_{i+1/2,j,k}-U^{n+1}_{i-1/2,j,k} \right) f^{n+1}_{i,j,k}+\frac{1}{\Delta y} \left( V^{n+1}_{i,j+1/2,k}-V^{n+1}_{i,j-1/2,k} \right) f^{n+1}_{i,j,k}\right]\\
        -\sum_{i,j,k}\frac{\gamma\Delta\xi}{\Delta \theta} \left(W^{n+1}_{i,j,k+1/2}-W^{n+1}_{i,j,k-1/2} \right) f^{n+1}_{i,j,k}.
\end{multline}
    We now use summation by parts \cref{lem:ibp} for the discrete Laplacians, the $\Pe$ and $\gamma$ drift terms. This gives 
    \begin{equation}\label{eq:step1}
        \begin{split}
            \sum_{i,j,k}\frac{\Delta\xi}{\Delta t} \left( f^{n+1}_{i,j,k}-f^n_{i,j,k} \right) f^{n+1}_{i,j,k}  = &- \sum_{i,j,k}\Delta\xi \left [D_T \left(|\dd_x f_{i+1/2,j,k}^{n+1}|^2+|\dd_y f_{i,j+1/2,k}^{n+1}|^2 \right) + |\dd_\theta f^{n+1}_{i,j,k+1/2}|^2 \right ]\\
        &  + \Pe \sum_{i,j,k}\Delta\xi\left[U^{n+1}_{i+1/2,j,k}\dd_xf^{n+1}_{i+1/2,j,k}+V^{n+1}_{i,j+1/2,k}\dd_y f^{n+1}_{i,j+1/2,k}\right]\\
        &  + \gamma \sum_{i,j,k}\Delta\xi W^{n+1}_{i,j,k+1/2}\dd_\theta f^{n+1}_{i,j,k+1/2}.
        \end{split}
    \end{equation}
    We note that, by \cref{eq:scheme_fluxes}, we have 
    \begin{equation}
        U^{n+1}_{i+1/2,j,k}\dd_x f^{n+1}_{i+1/2,j,k}\leq \tfrac{1}{2}(\cos(\theta_k)^++\cos(\theta_k)^-)\dd_x(f^{n+1}_{i+1/2,j,k})^2=\tfrac{1}{2}\cos(\theta_k)\dd_x(f^{n+1}_{i+1/2,j,k})^2,
    \end{equation} 
    and analogously for the $y$-direction. For the $\theta$-direction we have
    \begin{equation}
        \left|W^{n+1}_{i,j,k+1/2}\right|\leq \left|B[c^{n+1}]_{i,j,k+1/2}\right|f^{n+1}_{i,j,k}.
    \end{equation}
    Therefore,  from \cref{eq:step1} we obtain
        \begin{equation}
    \label{eq:intermL2-1}
        \begin{split}
            \sum_{i,j,k}\frac{\Delta\xi}{\Delta t}\left(f^{n+1}_{i,j,k}-f^n_{i,j,k}\right)f^{n+1}_{i,j,k}\leq &-\sum_{i,j,k}\Delta\xi\left[D_T\left(\left|\dd_x f_{i+1/2,j,k}\right|^2+\left|\dd_y f_{i,j+1/2,k}\right|^2\right)+\left|\dd_\theta f^{n+1}_{i,j,k+1/2}\right|^2\right]\\
                    & + \frac{1}{2}\Pe\sum_{i,j,k}\Delta\xi \left[\cos(\theta_k)\dd_x\left(f^{n+1}_{i+1/2,j,k}\right)^2+\sin(\theta_k)\dd_y\left(f^{n+1}_{i,j+1/2,k}\right)^2\right]\\
                    &+ \frac{1}{2}\gamma \sum_{i,j,k}\Delta\xi \left|B[c^{n+1}]_{i,j,k+1/2}\right|\dd_\theta \left(f^{n+1}_{i,j,k+1/2}
                    \right)^2 
        \end{split}
    \end{equation}

    We observe that the sum over the terms premultiplied by $\Pe$ in \cref{eq:intermL2-1} equals zero. Concerning the terms multiplied by $\gamma$, integrating by parts yields
    \begin{align}
        \frac\gamma2\sum_{i,j,k}\Delta\xi \left|B[c^{n+1}]_{i,j,k+1/2}
        \right|\,  \dd_\theta \left|f^{n+1}_{i,j,k+1/2}\right|^2
        &= -\frac\gamma2  \sum_{i,j,k}\Delta\xi\dd_\theta \left|B[c^{n+1}]_{i,j,k}\right|\,  \left|f^{n+1}_{i,j,k+1/2}\right|^2\\
        &\leq \frac\gamma2 \sum_{i,j,k}\Delta\xi \left| \dd_\theta \left|B[c^{n+1}]_{i,j,k}\right|\right|\, \left|f^{n+1}_{i,j,k}\right|^2\\
        &= \int \left|\dd_\theta B[c_h(t^{n+1})]\right| \, \left|f_h(t^{n+1})\right|^2 \dd \xi,  
    \end{align}
    having used the reverse triangle inequality, $\big||a|-|b|\big|\leq |a-b|$ for $a,b\in\mathbb{R}$. Then, the second sum in \cref{eq:intermL2-1} can be estimated as
    \begin{multline}
    \label{eq:intermL2-2}
\frac{1}{2} \Pe\sum_{i,j,k}\Delta\xi  \left[\cos(\theta_k)\dd_x\left(f^{n+1}_{i+1/2,j,k}\right)^2+\sin(\theta_k)\dd_y\left(f^{n+1}_{i,j+1/2,k}\right)^2\right]  \\+ \frac{1}{2}  \gamma \sum_{i,j,k}\Delta\xi  \left|B[c^{n+1}]_{i,j,k+1/2}\right|\dd_\theta \left(f^{n+1}_{i,j,k+1/2}\right)^2 \\
            \leq  \frac\gamma2 \int \left|\dd_\theta B[c_h(t^{n+1})]\right| \, \left|f_h(t^{n+1})\right|^2 \dd \xi. 
    \end{multline}
    Combining \eqref{eq:intermL2-2} and \eqref{eq:intermL2-1} and using $(b-a)b\geq\tfrac{1}{2}(b^2-a^2)$ on the left-hand of \eqref{eq:intermL2-1} for the time-derivative, we obtain
    \begin{align}
    \label{eq:intermL2-3}
        \begin{split}
            \frac{1}{2\Delta t} \left(\|f_h(t^{n+1})\|_{L^2}^2 - \|f_h(t^n)\|_{L^2}^2\right) 
            &\leq - D_T \|\dd_\xx f_h(t^{n+1})\|_{L^2}^2 - \|\dd_\theta f_h(t^{n+1})\|_{L^2}^2\\
            & \phantom{{} \leq} + \frac\gamma2\int \left|\dd_\theta B[c_h(t^{n+1})]\right| \, \left|f_h(t^{n+1})\right|^2 \dd \xi.
        \end{split}
    \end{align}
    Applying the  Cauchy--Schwarz and H\"older inequality, we can estimate the last term in \cref{eq:intermL2-3} as
    \begin{align}
    \label{eq:BcL2}
    \begin{split}
        \frac\gamma2\int \left|\dd_\theta B[c_h(t^{n+1})]\right| \, \left|f_h(t^{n+1})\right|^2 \dd \xi
        &\leq \|\dd_\theta B[c_h(t^{n+1})]\|_{L^2}\|f_h(t^{n+1})\|_{L^4}^2
        \\
        &\leq \|\dd_\theta B[c_h(t^{n+1})]\|_{L^2}\|f_h(t^{n+1})\|_{L^2}^{\frac{3}{2}}\|f_h(t^{n+1})\|_{L^6}^{\frac{1}{2}}.
    \end{split}
    \end{align}
    By the discrete Poincar\'e--Sobolev inequality \cite[Theorem 3.2]{bessemoulin2015discrete}, we have $\|f_h\|_{L^6}\leq C_{PS}\|f_h\|_{1,2}=C_{PS}(\|f_h\|_{L^2}+\|\dd_\xi f_h\|_{L^2})$, for some constant $C_{PS}>0$, which is independent of the mesh size. Therefore, we arrive at the inequality
     \begin{equation}\label{ineq:L2pregronwall}
        \frac{1}{2\Delta t}\left(\|f_h(t^{n+1})\|_{L^2}^2-\|f_h(t^{n})\|_{L^2}^2\right)\leq -D(\varepsilon)\|\dd_\xi f_h(t^{n+1})\|_{L^2}^2+C(\varepsilon)\|\dd_\theta B[c_h(t^{n+1})]\|_{L^2}^2\|f_h(t^{n+1})\|_{L^2}^2,
    \end{equation}
    where $\varepsilon>0$ can be made arbitrarily small at the cost of increasing $C(\varepsilon)>0$ and decreasing $D(\varepsilon)>0$. Hence, multiplying by $\Delta t$ and summing in time gives the result, provided $\Delta t < C' C_T$, with \begin{equation}C'=\tfrac{1}{2}C(\varepsilon)C_B.\end{equation}
\end{proof}
\begin{cor}\label{cor:unifL2rho}
    Provided the conditions for \cref{cor:unifL2} hold, there exists a constant $C>0$, independent of the mesh size, such that the solutions to the numerical scheme \eqref{eq:scheme} satisfy
    \begin{equation}
        \sup_{t\in[0,T]}\|\rho_h(t)\|_{L^2(\Omega)}\leq C.
    \end{equation}
\end{cor}
\begin{proof}
    The result follows by applying Jensen's inequality to $\|\rho_h(t)\|_{L^2(\Omega)}^2$, such that, if $t\in[t^{n-1},t^{n})$,
    \begin{align}
        \|\rho_h(t)\|_{L^2(\Omega)}^2& =\sum_{i,j}\left(\rho_{i,j}^n\right)^2\Delta\xx\\
        &\leq\sum_{i,j}\left(\sum_k f^n_{i,j,k}\Delta\theta\right)^2\Delta\xx\\
        &\leq\sum_{i,j,k}\left(f^n_{i,j,k}\right)^2\Delta\xi.
    \end{align}
    This concludes the proof.
\end{proof}

\section{Convergence of the scheme}\label{sec:conv}

The estimates from \cref{sec:fintest} allow us to derive a convergence result on a finite time interval $[0,T]$ as stated in \cref{thm:conv}, using a compactness method, similar to \cite{filbet2006finite}.
\begin{lem}\label{lem:h4}
    For all sequences $(f_h)_h$ of numerical solutions to the scheme as in \cref{def:scheme}, there exists a $C>0$ such that for all $s>0$ and $z\in\Sigma$,
    \begin{equation}\label{eq:timeshift}
        \int_0^{T-\tau}\int_\Sigma\left[f_h(t+s,\xi)-f_h(t,\xi)\right]\varphi(t,\xi)\dd\xi\dd t\leq Cs\|\varphi\|_{L^2(0,T;H^4(\Sigma))},
    \end{equation}
    and 
    \begin{equation}
        \int_0^T\int_{\Sigma-z}|f_h(t,\xi+z)-f_h(t,\xi)|^2\dd\xi\dd t\leq C|z|,
    \end{equation}
    where $\Sigma-z=\{\xi-z:\xi\in\Sigma\}\cap\Sigma$.
\end{lem}
\begin{proof}
    Following \cite[Proof of Proposition 4.1]{filbet2006finite}, we begin by showing that there exists $C>0$, such that
    \begin{equation}
        \|f_h\|_{L^3(\Sigma_T)}\leq C.
    \end{equation}
    To this end, we use the discrete $\mathrm{TV}$ inequality \cite[Theorem 2.4.1]{bessemoulin2015discrete}, so that for $u=|f_h|^{2}$ and a constant $C_{\mathrm{TV}}>0$, we have, using the notation $\bar{u}=\tfrac{1}{|\Sigma|}\int_\Sigma u\dd\xi$,
    \begin{equation}
        \|f_h\|_{L^3}^3=\|u\|_{L^{\frac{3}{2}}}^{\frac{3}{2}}\leq\left(\|u-\bar{u}\|_{L^{\frac{3}{2}}}+\|\bar{u}\|_{L^{\frac{3}{2}}}\right)^{\frac{3}{2}}\leq\left(C_{\mathrm{TV}}|u|_{1,1}+\|\bar{u}\|_{L^{\frac{3}{2}}}\right)^{\frac{3}{2}}.
    \end{equation}
    Hence, using $(x+y)^{\frac{3}{2}}\leq \tfrac{3}{2}(x^{\frac{3}{2}}+y^{\frac{3}{2}})$ for $x,y\in\mathbb{R}_{\geq 0}$, we get
    \begin{align}
        \|f_h\|_{L^3}^3&\leq \tfrac{3}{2}\left(C_{\mathrm{TV}}^{\frac{3}{2}}\|\dd_\xi |f_h|^2\|_{L^1}^{\frac{3}{2}}+\|f_h\|_{L^2}^{3}\right) \leq C\left(\|\dd_\xi |f_h|^2\|_{L^1}^{\frac{3}{2}}+\|f_h\|_{L^2}^{3}\right).
    \end{align}
   Using $a^2-b^2=(a+b)(a-b)$, we estimate $\dd_\xi |f_h|^2$ 
    and get
    \begin{equation}
    \begin{split}
        \int_0^T\|f_h(t)\|_{L^3}^3\dd t&\leq C\int_0^T\left(\left(\|f_h(t)\|_{L^2}\|\dd_\xi f_h(t)\|_{L^2}\right)^{\frac{3}{2}}+\|f_h(t)\|_{L^2}^3\right)\dd t.
    \end{split}
    \end{equation}
    Finally, we use the H\"older inequality and the uniform $L^2$-bounds on $f_h$ and $\dd_\xi f_h$ from \cref{cor:unifL2} so that there is a constant $C>0$ independent of $h$ such that $\|f_h\|_{L^3(\Sigma_T)}\leq C$.

    Now turning to the time translate, for $\varphi\in L^2(0,T;H^4(\Sigma))$, we get, defining $M\Delta t\leq s<(M+1)\Delta t$ for $M\in\mathbb{N}$, and for $\varphi^n_{i,j,k}=\int_{t^{n-1}}^{t^n}\int_{C_{i,j,k}}\varphi(t,\xi)\dd\xi\dd t$,
    \begin{equation}
        \begin{split}
            I_s=&\int_0^{T-s}\int_\Sigma\left(f(t+s,\xi)-f(t,\xi)\right)\varphi(t,\xi)\dd\xi\dd t\\=&\sum_{n=1}^{N_T-M}\sum_{i,j,k}\left(f^{n+M}_{i,j,k}-f^n_{i,j,k}\right)\varphi^n_{i,j,k}\Delta \xi\Delta t\\
            =&\sum_{n=1}^{N_T-M}\sum_{i,j,k}\sum_{m=1}^M\frac{1}{\Delta t}\left(f^{n+m}_{i,j,k}-f^{n+m-1}_{i,j,k}\right)\varphi^n_{i,j,k}\Delta \xi\Delta t^2.
        \end{split}
    \end{equation}
    We insert the scheme as in \cref{eq:scheme} into $I_s$ to obtain, after summation by parts,
    \begin{align}
        I_s=
        &-D_T \sum_{n=1}^{N_T-M}\sum_{i,j,k}\sum_{m=1}^M \left(\dd_x f^{n+m}_{i+1/2,j,k}\dd_x\varphi^n_{i+1/2,j,k}+\dd_y f^{n+m}_{i,j+1/2,k}\dd_y\varphi^n_{i,j+1/2,k}\right)\Delta\xi\Delta t^2\\
        &+\Pe\sum_{n=1}^{N_T-M}\sum_{i,j,k}\sum_{m=1}^M\left(U^{n+m}_{i+1/2,j,k}\dd_x\varphi^n_{i+1/2,j,k}+V^{n+m}_{i,j+1/2,k}\dd_y\varphi^n_{i,j+1/2,k}\right)\Delta\xi\Delta t^2\\
        &- \sum_{n=1}^{N_T-M}\sum_{i,j,k}\sum_{m=1}^M \dd_\theta f^{n+m}_{i,j,k+1/2}\dd_\theta\varphi^n_{i,j,k+1/2}\Delta\xi\Delta t^2+I_s^\gamma,
    \end{align}
    where $I^\gamma_s$ is defined as
    \begin{equation}
        I_s^\gamma =\gamma\sum_{n=1}^{N_T-M}\sum_{i,j,k}\sum_{m=1}^M W^{n+m}_{i,j,k+1/2} \dd_\theta\varphi^n_{i,j,k+1/2} \Delta\xi\Delta t^2.
    \end{equation}
    It is clear that, except for the $I^\gamma_s$ term, using Cauchy--Schwarz, we get, for some constant $C>0$, using the uniform $L^2$-bounds on $f_h$ and $\dd_\xi f_h$ from \cref{cor:unifL2},
    \begin{equation}
        |I_s|\leq C s\|\varphi\|_{L^2(0,T;H^4(\Sigma))}+|I_s^\gamma|.
    \end{equation}
    For the $I^\gamma_s$ term we get
    \begin{equation}
        |I_s^\gamma|\leq\gamma\sum_{n=1}^{N_T-M}\sum_{i,j,k}\sum_{m=1}^M|B[c^{n+m}]_{i,j,k+1/2}|(|f_{i,j,k}^{n+m
        }|+|f_{i,j,k+1}^{n+m}|)\dd_\theta\varphi^n_{i,j,k+1/2}\Delta\xi\Delta t^2.
    \end{equation}
    By the definitions for $B\in\{B_0,B_\lambda,B_\tau\}$ and the discrete elliptic estimates on $c_h$ in terms of $L^2(\Omega)$-norms of $\rho_h$ (cf.~\cref{cor:disccreg}), we obtain, by \cref{cor:unifL2rho}, 
    \begin{equation}
        \|B[c_h(t)]\|_{L^2}\leq C_B\|\rho_h(t)\|_{L^2(\Omega)}\leq C,
    \end{equation}
    for some $C>0$. Hence, by using H\"older's inequality, we can estimate, for some constant $C>0$,
    \begin{equation}
    \begin{split}
        |I^\gamma_s|
        &\leq 2\gamma \sum_{m=1}^M \|B[c_h]\|_{L^6(0,T;L^2(\Sigma))}\|f_h\|_{L^3(\Sigma_T)}\|\dd_\theta\varphi_h\|_{L^2(0,T;L^6(\Sigma))}\Delta t\\
        &\leq  Cs\|\dd_\theta\varphi_h\|_{L^2(0,T;L^6(\Sigma))}\\
       &\leq  Cs\|\varphi\|_{L^2(0,T;H^4(\Sigma))},
    \end{split}
    \end{equation}
    where we used \cite[Theorem 3 in Section 5.8.2]{evans2010partial} for the difference quotient estimate $\|\dd_\theta\varphi_h\|_{L^6(\Sigma)}\leq C\|\partial_\theta\varphi\|_{L^6(\Sigma)}$ for some constant $C>0$, and we used the Sobolev embedding in the last line.
    
    Finally, the phase-space translate result can be obtained in a similar fashion as in a classical theorem for finite volume schemes \cite[Lemma 9.3]{eymard2000finite}. Indeed, let $(i_z,j_z,k_z)$ be such that $z\in C_{i_z,j_z,k_z}$, then
    \begin{align}
        &\int_0^T\int_{\Sigma-z}\left|f_h(t,\xi+z)-f_h(t,\xi)\right|^2\dd\xi\dd t=\sum_{n=1}^{N_T}\sum_{\substack{i,j,k:\\ C_{i,j,k}\subseteq\Sigma-z}}\left|f^n_{i+i_z,j+j_z,k+k_z}-f^n_{i,j,k}\right|^2\Delta \xi\Delta t\\
        &=\sum_{n=1}^{N_T}\sum_{\substack{i,j,k:\\ C_{i,j,k}\subseteq\Sigma-z}}\left(f^n_{i+i_z,j+j_z,k+k_z}-f^n_{i,j,k}\right)\times\\& \quad \Bigg[\sum_{i'=i}^{i+i_z-1} \dd_x f^n_{i'+1/2,j,k}\Delta x+\sum_{j'=j}^{j+j_z-1}\dd_y f^n_{i+i_z,j'+1/2,k}\Delta y+\sum_{k'=k}^{k+k_z-1}\dd_\theta f^n_{i+i_z,j+j_z,k'+1/2}\Delta\theta\Bigg]\Delta \xi\Delta t.
    \end{align}
    Hence, by Cauchy--Schwarz, we get
    \begin{equation}
    \begin{split}
        \int_0^T\int_{\Sigma-z}|f_h(t,\xi+z)-f_h(t,\xi)|^2\dd\xi\dd t&\leq C|z|\|f_h\|_{L^2(\Sigma_T)}\|\dd_\xi f_h\|_{L^2(\Sigma_T)} \leq C|z|,
        \end{split}
    \end{equation}
    having used \cref{cor:unifL2}. This concludes the proof.
\end{proof}
\begin{lem}\label{lem:comp}
    For all sequences $(f_h)_h$ of numerical solutions to the scheme as in \cref{def:scheme}, there is a subsequence that converges strongly to a function $f\in L^2(\Sigma_T)$, the discrete gradient $(\dd_\xi f_h)_h$ converges weakly in $L^2(\Sigma_T)$ to $\nabla_\xi f \in L^2(\Sigma_T)$ and $B[c_h]$ converges weakly-$\ast$ in $L^\infty(0,T;L^2)$ to $B[c]$ such that $c$ is the unique strong solution for the elliptic equation with $\rho=\int_0^{2\pi} f\dd\theta$ (possibly up to another subsequence).
\end{lem}
\begin{proof}
    By the Banach--Alaoglu theorem and the uniform bounds of \cref{cor:unifL2} on $f_h$ in $L^2(\Sigma_T)$ we can extract a subsequence that converges weakly in $L^2(\Sigma_T)$ to some $f\in L^2(\Sigma_T)$. We first show that this convergence is strong in $L^2(\Sigma_T)$.
    
    By \cref{lem:h4} we have that $(f_h)_h$ is bounded in $L^2(0,T;X_3(\Sigma))$, where $X_3(\Sigma)$ is as in \cite[Section 4]{filbet2006finite}, and that \begin{equation}
        \|f_h(t+s)-f_h(t)\|_{L^2(0,T;(H^4(\Sigma))')}\leq Cs.
    \end{equation} 
    We also have the continuous embeddings 
    \begin{equation}
        X_3(\Sigma)\subset L^2(\Sigma)\subset (H^4(\Sigma))',
    \end{equation}
    and $X_3(\Sigma)\subset L^2(\Sigma)$ is a compact embedding by the Riesz--Kolmogorov--Fr\'echet theorem \cite[Theorem 4.26]{brezis2010functional}. Hence, by the Aubin--Lions--Simon theorem \cite[Theorem 5]{simon1986compact}, we have that $f_h\to f$ strongly in $L^2(\Sigma_T)$ (up to a subsequence).
    
    From the Banach-Alaoglu theorem, we derive that $\dd_\xi f_h\rightharpoonup \nabla_\xi f$, weakly in $L^2(\Sigma_T)$. 
    
    By the uniform bounds in $L^\infty(0,T;L^2(\Omega))$ for $c_h,\dd_\xx c_h$ and $\dd_\xx^2 c_h$ of \cref{cor:disccreg}, it follows that there are functions $c,\mathbf{g}$ and $\mathbf{M}$ in $L^\infty(0,T;L^2(\Omega))$ such that there is a subsequence and there is component-wise convergence 
    \begin{equation}
        c_h\stackrel{\ast}{\rightharpoonup} c,\quad \dd_\xx c_h \stackrel{\ast}{\rightharpoonup} \mathbf{g},\quad \dd_\xx^2 c_h\stackrel{\ast}{\rightharpoonup}\mathbf{M},
    \end{equation}
    weakly-$\ast$ in $L^\infty(0,T;L^2(\Omega))$.
    Similar as in \cite[Lemma 4.4]{chainais2003finite}, it follows that $\mathbf{g}=\nabla_\xx c$ and $\mathbf{M}=\nabla_\xx^2 c$ in distribution.
    
    Following the argument in \cite[Proposition 4.2]{filbet2006finite}, we next show that the limit $c$ satisfies the elliptic equation $0=\Delta c-\alpha c+\rho$, where $\rho$ is the limit of $\rho_h$, and $\rho=\int_0^{2\pi} f(t,\xi)\dd\theta$ by the convergence of $f_h$. To this end,  let $\varphi\in C^\infty(\overline{\Omega_T})$ be given and define
    \begin{align}
        I_h=&\int_{\Omega_T}\left(-\dd_\xx c_h\cdot\nabla_\xx\varphi+(\rho_h-\alpha c_h)\varphi\right)\dd\xx\dd t,
    \end{align}
    and show that 
    \begin{align}
        \label{eq:chcrho}
        I_h \to \int_{\Omega_T}\left(-\nabla_\xx c\cdot\nabla_\xx\varphi+(\rho-\alpha c)\varphi\right)\dd\xx\dd t, 
    \end{align}
    as well as 
    \begin{align}
        \label{eq:chcrhoscheme}
        I_h \to \int_{\Omega_T}\left(-\dd_\xx c_h\cdot\nabla_\xx\varphi+(\rho_h-\alpha c_h)\varphi\right)\dd\xx\dd t \to 0,
    \end{align}
    as $h\to 0$.

    The first convergence in \cref{eq:chcrho} follows from the weak-$\ast$ convergence. For the second convergence, \cref{eq:chcrhoscheme}, we multiply the equation for $c_h$,  \cref{eq:discc}, by $\varphi^n_{i,j}=\int_{t^{n-1}}^{t^n}\int_{C_{i,j}}\varphi(t,\xx)\dd\xx \dd t$ and  $\Delta \xx, \Delta t$,  and sum integrate over all $i,j,n$ to get
    \begin{align}
        I_h^S = \sum_{n=0}^{N_T-1}\sum_{i,j} \left(\dd_x^2 c_{i,j}^n + \dd_y^2 c_{i,j}^n \right)  \varphi_{i,j}^n \Delta \xx \Delta t + \sum_{n=0}^{N_T-1}\sum_{i,j}  ( - \alpha c_{i,j}^n+\rho_{i,j}^n) \varphi_{i,j}^n \Delta \xx \Delta t = 0.
    \end{align}
    Then, a discrete integration by parts yields
        \begin{align}
            I_h^S 
            &=\sum_{n=0}^{N_T-1}\sum_{i,j}\left(\dd_x c_{i+1/2,j}^{n}\dd_x\varphi^n_{i+1/2,j}+\dd_y c_{i,j+1/2}^{n}\dd_y\varphi^n_{i,j+1/2}\right)\Delta\xx\Delta t\\
            &\qquad -\sum_{n=0}^{N_T-1}\sum_{i,j}(\rho^{n}_{i,j}-\alpha c^{n}_{i,j})\varphi^n_{i,j}\Delta\xx\Delta t.
        \end{align}
    We now show that $I_h-I_h^S\to 0$, as $h\to 0$. Using the mean value theorem, \cref{cor:unifL2rho} and \cref{cor:disccreg} and writing $\varphi^n_{i,j}=\varphi(t^n,\xx_{i,j})$,
    \begin{equation}
        \begin{split}
            &\int_{\Omega_T}\left(\rho_h-\alpha c_h\right)\varphi\dd\xx\dd t-\sum_{n=0}^{N_T-1}\sum_{i,j}(\rho^{n}_{i,j}-\alpha c^{n}_{i,j})\varphi^n_{i,j}\Delta\xx\Delta t\\
            =&\sum_{n=0}^{N_T-1}\sum_{i,j}\int\limits_{t^n}^{t^{n}}\int\limits_{C_{i,j}}\left(\rho^{n}_{i,j}-\alpha c^{n}_{i,j}\right)
            \left(\varphi(t,\xx) - \varphi^n_{i,j}\right)\dd\xx\dd t\\
            \to& \ 0.
        \end{split}
    \end{equation}
    Similarly, for the gradient terms, we have
    \begin{equation}
        \begin{split}
            &\int_{\Omega_T}\dd_x c_h\partial_x \varphi\dd\xx\dd t-\sum_{n=0}^{N_T-1}\sum_{i,j}\dd_x c_{i+1/2,j}^{n}\dd_x\varphi^n_{i+1/2,j}\Delta\xx\Delta t\\
            =&\sum_{n=0}^{N_T-1}\sum_{i,j}\dd_x c^{n}_{i+1/2,j}
            \left(\int_{t^{n-1}}^{t^n}\int_{C_{i+1/2,j}}\partial_x\varphi - \dd_x\varphi^n_{i+1/2,j} \dd\xx\dd t \right)\\
            \to& \ 0.
        \end{split}
    \end{equation}
    In conclusion, we have shown that
    \begin{equation}
            I_h-I_h^S\to 0.
    \end{equation}
    This implies that $c$ solves the elliptic equation on $\Omega_T$. Similarly, using the $L^\infty(0,T;L^2(\Omega))$ bounds on $c_h$ and $\dd_\xx c_h$ we can also show that $c$ solves the elliptic equation pointwise in time and in the strong sense.
    
    Finally, for each $B\in\{B_0,B_\lambda,B_\tau\}$ we have $B[c_h]\stackrel{\ast}{\rightharpoonup} B[c]$ weakly-$\ast$ in $L^\infty(0,T;L^2(\Sigma))$. Indeed, we have the following computations.

    First we rewrite the difference $B[c_h]-B[c]$ as follows. For $(x,y,\theta)\in C_{i,j,k+1/2}$ and $t\in (t^{n-1}, t^n]$, we observe that
    \begin{align}
        B_0[c_h] - B_0[c] = & \mathbf{n}(\theta_{k+1/2})\cdot D c_{i,j}^n - \mathbf{n}(\theta)\cdot\nabla_\xx c \\
        = & \left(\mathbf{n}(\theta_{k+1/2})-\mathbf{n}(\theta)\right)\cdot Dc_{i,j}^n-\mathbf{n}(\theta)\cdot\left(\nabla_\xx c-Dc_{i,j}^n\right).
    \end{align}
    as well as 
    \begin{align}
        B_\lambda[c_h] - B_\lambda[c] = & \ \mathbf{n}(\theta_{k+1/2})\cdot\discgrad c_{i,j,k+1/2}^n-\mathbf{n}(\theta)\cdot\nabla_\xx c_\lambda\\
        = & \ \left(\mathbf{n}(\theta_{k+1/2})-\mathbf{n}(\theta)\right)\cdot Dc_{i,j,k+1/2}^n-\mathbf{n}(\theta)\cdot\left(\nabla_\xx c_\lambda-Dc_{i,j,k+1/2}^n\right)
    \end{align}
    and
    \begin{align}
        B_\tau[c_h] - B_\tau[c] & = \ \ \ \mathbf{n}(\theta_{k+1/2})\cdot\discgrad c_{i,j}^n+\tau\mathbf{n}(\theta_{k+1/2})\cdot\dischess c_{i,j}^n\mathbf{e}(\theta_{k+1/2})-\mathbf{n}(\theta)\cdot\nabla c-\tau\mathbf{n}(\theta)\nabla_\xx^2 c\mathbf{e}(\theta)\\
        & =  \ \left(\mathbf{n}(\theta_{k+1/2})-\mathbf{n}(\theta)\right)\cdot Dc_{i,j}^n-\mathbf{n}(\theta)\cdot\left(\nabla_\xx c-Dc_{i,j}^n\right)\\
        &\quad + \ \tau\left(\mathbf{e}(2\theta_{k+1/2})-\mathbf{e}(2\theta)\right)\cdot \tilde{H}c_{i,j}^n-\tau\mathbf{e}(2\theta)\cdot\left(\tilde{\nabla}_\xx^2 c-\tilde{H}c_{i,j}^n\right),
    \end{align}
    and where we used the notation 
    \begin{equation}
        \tilde{H}c_{i,j}=\left(D_xD_y c_{i,j},\tfrac{1}{2}\left(\dd_y^2 c_{i,j}-\dd_x^2 c_{i,j}\right)\right)^\mathsf{T}, \quad \tilde{\nabla}_\xx^2 c=\left(\partial_{xy} c,\tfrac{1}{2}(\partial_y^2 c-\partial_x^2 c)\right)^{\mathsf{T}},
    \end{equation}
    to rewrite the Hessian term, and the identity
    \begin{equation}
        \mathbf{n}(\theta)\cdot \begin{pmatrix}
            a & c\\
            c & d
        \end{pmatrix}\mathbf{e}(\theta)=\mathbf{e}(2\theta)\cdot\begin{pmatrix}
            c\\ \tfrac{1}{2}(d-a)
        \end{pmatrix}.
    \end{equation}
    Now, let us begin by addressing the term $B_0$. To this end, we let $\varphi\in L^1(0,T;L^2(\Sigma))$ and observe that
    \begin{align}
        \mathcal{I}_h&= \int_{\Sigma_T}\left(B_0[c_h]-B_0[c]\right)\varphi\dd\xi\dd t\\
        &=\sum_{n=1}^{N_T}\int_{t^{n-1}}^{t^n}\sum_{i,j,k} \int_{C_{i,j,k+1/2}}\left(\mathbf{n}(\theta_{k+1/2})\cdot Dc_{i,j}^n-\mathbf{n}(\theta)\cdot\nabla_\xx c\right)\varphi\dd\xi\dd t\\
        &=\sum_{n=1}^{N_T}\int_{t^{n-1}}^{t^n}\sum_{i,j,k} \int_{C_{i,j,k+1/2}}\left(\left(\mathbf{n}(\theta_{k+1/2})-\mathbf{n}(\theta)\right)\cdot Dc_{i,j}^n - \mathbf{n}(\theta)\cdot\left(\nabla_\xx c-Dc_{i,j}^n\right)\right)\varphi\dd\xi\dd t.
    \end{align}
    Passing to the modulus, we get, for some $C>0$,
    \begin{align}
        \left|\mathcal{I}_h\right|&\leq \ \ C\Bigg(\|\mathbf{n}_h - \mathbf{n}\|_{L^\infty}\|\dd_\xx c_h\|_{L^\infty(0,T;L^2(\Omega))}\|\varphi\|_{L^1(0,T;L^2)}\\
        &\qquad\qquad +\left|\int_{\Sigma_T}\left(\partial_x c-\dd_x c_h\right)\sin(\theta)\varphi\dd\xi\dd t\right|+\left|\int_{\Sigma_T}\left(\partial_y c-\dd_y c_h\right)\cos(\theta)\varphi\dd\xi\dd t\right|\Bigg),
    \end{align}
    where we use the notation $\mathbf{n}_h$ to denote the piecewise-constant interpolation, such that $\mathbf{n}_h(\xi)=\mathbf{n}(\theta_{k+1/2})$ when $\xi\in C_{i,j,k+1/2}$.

    By the mean-value theorem, we have
    \begin{equation}
        \|\mathbf{n}_h-\mathbf{n}\|_{L^\infty}\leq h.
    \end{equation}
    We also have that $\sin(\theta)\varphi,\cos(\theta)\varphi\in L^1(0,T;L^2)$, and $\dd_x c_h\stackrel{\ast}{\rightharpoonup} \partial_x c$ and $\dd_y c_h\stackrel{\ast}{\rightharpoonup}\partial_y c$ weakly-$\ast$ in $L^\infty(0,T;L^2(\Omega))$. Hence, as $h\to 0$, we have $|\mathcal{I}_h|\to 0$, i.e.,
    \begin{equation}
         B_0[c_h]\stackrel{\ast}{\rightharpoonup}B[c], \quad \mathrm{ in } \ L^\infty(0,T;L^2).
    \end{equation}
    Similarly, we can show that 
    \begin{equation}
        B_\lambda[c_h]\stackrel{\ast}{\rightharpoonup}B_\lambda[c], \quad B_\tau[c_h]\stackrel{\ast}{\rightharpoonup} B_\tau[c], \quad \mathrm{ in } \ L^\infty(0,T;L^2).
    \end{equation}
    This concludes the proof.
\end{proof}
\begin{proof}[Proof of Theorem 2.1]
    Let $\varphi\in C^\infty(\overline{\Sigma_T})$ such that $\varphi(T)=0$ and define the error term $\epsilon(h)$ as 
    \begin{equation}
        \label{eq:error}
        \begin{split}
            \epsilon(h)=&-\int_0^T\int_\Sigma f_h\partial_t\varphi\dd\xi\dd t-\int_0^{T}\int_{\Sigma} (D_T\dd_\xx f_h-\Pe\mathbf{e}_\theta f_h)\cdot\nabla_\xx\varphi + (\dd_\theta f_h-\gamma B[c_h]f_h) \, \partial_\theta\varphi\dd\xi\dd t\\
            &-\int_\Sigma f_h(0)\varphi(0)\dd\xi,
        \end{split}
    \end{equation}
    By \cref{lem:comp}, we have that $\epsilon(h)$ converges to the left-hand side of \cref{eq:weakeq}, that is, 
    \begin{equation}
        \begin{split}
            \epsilon(h)\to& -\int_0^T\int_\Sigma \left[f\partial_t\varphi -(D_T\nabla_\xx f-\Pe\mathbf{e}_\theta f)\cdot\nabla_\xx\varphi-(\partial_\theta f-\gamma B[c] f)\partial_\theta\varphi \right]\dd\xi\dd t\\&-\int_\Sigma f(0)\varphi(0)\dd\xi.
        \end{split}
    \end{equation}
    It remains to prove that $\epsilon(h)\to 0$ for the subsequence of \cref{lem:comp}, so that the limit $f$ is indeed a weak solution as in \cref{def:solf}.

    To do this, we compare $\epsilon(h)$ with the scheme. We multiply the scheme as in \cref{eq:scheme} by $\varphi^{n+1}_{i,j,k}$, where 
    \begin{equation}
        \label{eq:defn-disc-test-fun}
        \varphi^n_{i,j,k}=\frac{1}{\Delta\xi}\int_{C_{i,j,k}}\varphi(t^n,\xi)\dd\xi,
    \end{equation}
    and sum from $n=0$ to $N_T$ and over all cells so that  
    \begin{equation}
    \begin{split}
        0=\sum_{n=0}^{N_T}\sum_{i,j,k}\Bigg[&\frac{f^{n+1}_{i,j,k}-f^n_{i,j,k}}{\Delta t}+\dd_x F^{x,n+1}_{i,j,k}+\dd_y F^{y,n+1}_{i,j,k}
        +\dd_\theta F^{\theta,n+1}_{i,j,k}\Bigg]\varphi^{n+1}_{i,j,k}\Delta\xi\Delta t.
    \end{split}
    \end{equation}
    By summation by parts and $\varphi(T)=0=\varphi^{N_T+1}_{i,j,k}$ we obtain
    \begin{equation}
    \begin{split}
         0=&-\sum_{i,j,k}\left[\sum_{n=1}^{N_T}f^{n+1}_{i,j,k}(\varphi^{n+1}_{i,j,k}-\varphi^{n}_{i,j,k}) +\varphi^{1}_{i,j,k}f^{0}_{i,j,k}\right]\Delta\xi\\
         &-\sum_{n=0}^{N_T}\sum_{i,j,k}\left[\dd_x\varphi^{n+1}_{i+1/2,j,k}F^{x,n+1}_{i+1/2,j,k}+\dd_y\varphi^{n+1}_{i,j+1/2,k}F^{y,n+1}_{i,j+1/2,k}+\dd_\theta\varphi^{n+1}_{i,j,k+1/2}F^{\theta,n+1}_{i,j,k+1/2}\right]\Delta\xi\Delta t.
    \end{split}
    \end{equation}
    We compare each component of the scheme with its continuum counterpart. Similar as in \cite{bailo2020convergence}, we write
    \begin{equation}
        0=\hat{\mathcal{T}}(h)+\hat{\mathcal{D}}(h)+\hat{\mathcal{P}}(h)+\hat{\mathcal{U}}(h)+\hat{\mathcal{E}}(h),
    \end{equation}
    where 
    \begin{align}
        \hat{\mathcal{T}}(h)=&-\sum_{i,j,k}\left[\sum_{n=1}^{N_T}f^n_{i,j,k}(\varphi^{n+1}_{i,j,k}-\varphi^{n}_{i,j,k})+\varphi^1_{i,j,k}f^0_{i,j,k}\right]\Delta \xi,\\
        \hat{\mathcal{D}}(h)=&\quad \sum_{n=0}^{N_T}\sum_{i,j,k}\left[D_T(\dd_x\varphi^{n+1}_{i+1/2,j,k}\dd_x f^{n+1}_{i+1/2,j,k}+\dd_y\varphi^{n+1}_{i,j+1/2,k}\dd_y f^{n+1}_{i,j+1/2,k})+\dd_\theta\varphi^{n+1}_{i,j,k+1/2}\dd_\theta f^{n+1}_{i,j,k+1/2}\right]\Delta\xi\Delta t,\\
        \hat{\mathcal{P}}(h)=&-\sum_{n=0}^{N_T}\sum_{i,j,k}\Pe\left[\dd_x\varphi^{n+1}_{i+1/2,j,k}\cos(\theta_k)f^{n+1}_{i,j,k}+\dd_y\varphi^{n+1}_{i,j+1/2,k}\sin(\theta_k)f^{n+1}_{i,j,k}\right]\Delta\xi\Delta t,\\
        \hat{\mathcal{U}}(h)=&-\sum_{n=0}^{N_T}\sum_{i,j,k}\gamma\left[\dd_\theta\varphi^{n+1}_{i,j,k+1/2}B[c^{n+1}]_{i,j,k+1/2}f^{n+1}_{i,j,k}\right]\Delta\xi\Delta t,\\
        \hat{\mathcal{E}}(h)=& \quad \sum_{n=0}^{N_T}\sum_{i,j,k}\Pe\left[\Delta x(\cos(\theta_k))^-\dd_x f_{i+1/2,j,k}^{n+1}\dd_x\varphi^{n+1}_{i+1/2,j,k}+\Delta y(\sin(\theta_k))^-\dd_y f_{i,j+1/2,k}^{n+1}\dd_y\varphi^{n+1}_{i,j+1/2,k}\right]\Delta\xi\Delta t\\
        &+\sum_{n=0}^{N_T}\sum_{i,j,k}\gamma\left[\Delta\theta(B[c^{n+1}]_{i,j,k+1/2})^-\dd_\theta f^{n+1}_{i,j,k+1/2}\dd_\theta\varphi^{n+1}_{i,j,k+1/2}\right]\Delta\xi\Delta t.
    \end{align}
    Similarly, we write \cref{eq:error} as 
    \begin{equation}
        \epsilon(h)=\mathcal{T}(h)+\mathcal{D}(h)+\mathcal{P}(h)+\mathcal{U}(h),
    \end{equation}
    where 
    \begin{align}
        \mathcal{T}(h)=&-\int_\Sigma\left[\int_0^T f_h\partial_t \varphi \dd t+f_h(0,\xi)\varphi(0,\xi)\right]\dd\xi,\\
        \mathcal{D}(h)=&\quad \int_0^T\int_{\Sigma} [D_T\dd_\xx f_h\cdot\nabla_\xx\varphi+\dd_\theta f_h\partial_\theta\varphi]\dd\xi\dd t,\\
        \mathcal{P}(h)=&-\int_0^T\int_{\Sigma}\Pe f_h \mathbf{e}_h(\theta) \cdot\nabla_\xx\varphi\dd\xi\dd t,\\
        \mathcal{U}(h)=&-\int_0^T\int_{\Sigma}\gamma B[c_h]f_h\partial_\theta\varphi\dd\xi\dd t.
    \end{align}
We now compare the terms $\hat{\mathcal{T}},\hat{\mathcal{D}},\hat{\mathcal{P}}$ and $\hat{\mathcal{U}}$ with their continuum counterpart $\mathcal{T},\mathcal{D},\mathcal{P}$ and $\mathcal{U}$, respectively, and show that their difference converges to zero. We also show that $\hat{\mathcal{E}}(h)\to 0$, so that, in conclusion, $\epsilon(h)\to 0$.

Clearly, as in \cite[Proof of Theorem 2.1]{bailo2020convergence}, $\mathcal{T}(h)-\hat{\mathcal{T}}(h)=0$ for all $h$. Similarly, the term $\hat{\mathcal{E}}(h)$ goes to zero by using Cauchy-Schwarz, the uniform $L^2(\Sigma_T)$-bound on $\dd_\xi f_h$ and $B[c_h]$ and $\|\varphi\|_{C^2(\Sigma_T)}$. 

Indeed, using \cref{eq:defn-disc-test-fun}, we find
\begin{equation}
    \begin{split}
        \mathcal{T}(h)-\hat{\mathcal{T}}(h)=&\quad \sum_{i,j,k}\left(\sum_{n=1}^{N_T}f^n_{i,j,k}\left[\left(\varphi^{n+1}_{i,j,k}-\varphi^n_{i,j,k}\right)\Delta\xi-\int\limits_{C_{i,j,k}}\int\limits_{t^{n}}^{t^{n+1}}\partial_t \varphi\dd t\dd\xi\right]\right)\\&+\sum_{i,j,k}f^0_{i,j,k}\left[\varphi^1_{i,j,k}-\frac{1}{\Delta\xi}\int\limits_{C_{i,j,k}}\left(\int\limits_{t^0}^{t^1}\partial_t\varphi\dd t+\varphi(0,\xi)\right)\dd\xi\right]\Delta\xi\\
        =&0.
    \end{split}
\end{equation}
Next, let us address $\hat{\mathcal{E}}(h)$. Focusing on the term premultiplied by $\gamma$, and application of the Cauchy--Schwarz inequality, we have 
\begin{align}
    &\sum_{n=0}^{N_T} 
    \sum_{i,j,k} \gamma\left[\Delta\theta(B[c^{n+1}]_{i,j,k+1/2})^- \dd_\theta f^{n+1}_{i,j,k+1/2} \dd_\theta\varphi^{n+1}_{i,j,k+1/2}\right] \Delta\xi \Delta t\\
    &\quad\leq \Delta \theta \gamma \|\varphi\|_{C^1(\overline{\Sigma_T})}\|B[c_h]\|_{L^2(\Sigma_T)}\|\dd_\theta f_h\|_{L^2(\Sigma_T)}\\
    &\to 0.
\end{align}
The other terms in $\hat{\mathcal{E}}(h)$ go to zero by a similar argument.

Concerning the difference $\mathcal{D}(h)-\hat{\mathcal{D}}(h)$, we focus on the $x$-terms  noting that the terms corresponding to $y$ and $\theta$ are treated in the same manner. The $x$-terms in $\mathcal{D}(h)-\hat{\mathcal{D}}(h)$ are
\begin{equation}\label{eq:dh}
\begin{split}
    &D_T \left[\int_0^T\int_{\Sigma} \dd_x f_h\partial_x \varphi\dd\xi\dd t-\sum_{n=0}^{N_T}\sum_{i,j,k} \dd_x\varphi^{n+1}_{i+1/2,j,k}\dd_x f^{n+1}_{i+1/2,j,k}\Delta\xi\Delta t\right]\\
    =& D_T\sum_{n=1}^{N_T}\sum_{i,j,k}\left[\int_{t^{n-1}}^{t^{n}}\int_{C_{i+1/2,j,k}}\dd_x f_h\partial_x\varphi\dd\xi\dd t-\dd_x\varphi^{n}_{i+1/2,j,k}\dd_x f^{n}_{i+1/2,j,k}\Delta\xi\Delta t\right]\\
    =& D_T\sum_{n=1}^{N_T}\sum_{i,j,k}  \dd_x f^{n}_{i+1/2,j,k}\left[\int_{t^{n-1}}^{t^{n}}\int_{C_{i+1/2,j,k}} \partial_x\varphi\dd\xi\dd t-\dd_x\varphi^{n}_{i+1/2,j,k}\Delta\xi\Delta t\right]
    \end{split}
\end{equation}
where we use again $\varphi^{N+1}_{i,j,k}=0$. Furthermore, we can write the continuum integrals for each cell as 
\begin{equation}
\begin{split}
   &\int_{t^{n-1}}^{t^{n}}\int_{C_{i+1/2,j,k}}D_T\dd_x f_h\partial_x\varphi\dd\xi\dd t \\=& D_T\dd_x f^{n}_{i+1/2,j,k}\int\limits_{t^{n-1}}^{t^{n}}\int\limits_{y_{j-1/2}}^{y_{j+1/2}}\int\limits_{\theta_{k-1/2}}^{\theta_{k+1/2}}[\varphi(t,x_{i+1},y,\theta)-\varphi(t,x_{i},y,\theta)]\dd y\dd\theta\dd t.
\end{split}
\end{equation}
Hence, comparing using the mean value theorem, we get
\begin{equation}
    \label{eq:mvt}
    \begin{split}
        &\int\limits_{t^{n-1}}^{t^{n}}\int\limits_{y_{j-1/2}}^{y_{j+1/2}}\int\limits_{\theta_{k-1/2}}^{\theta_{k+1/2}}[\varphi(t,x_{i+1},y,\theta)-\varphi(t,x_{i},y,\theta)]\dd y\dd\theta\dd t-\dd_x\varphi^n_{i+1/2,j,k}\Delta\xi\Delta t\\
        =&\int\limits_{y_{j-1/2}}^{y_{j+1/2}}\int\limits_{\theta_{k-1/2}}^{\theta_{k+1/2}}\Big[\int\limits_{t^{n-1}}^{t^{n}}\frac{\varphi(t,x_{i+1},y,\theta)-\varphi(t,x_{i},y,\theta)}{\Delta x}\dd t \\& \quad\quad\quad\quad\quad\quad-\frac{\Delta t}{\Delta x}\int_{x_{i-1/2}}^{x_{i+1/2}}\frac{\varphi(t^n,\Delta x+s,y,\theta)-\varphi(t^n,s,y,\theta)}{\Delta x}\dd s\Big]\Delta x\dd y\dd\theta\\
        =&\int\limits_{y_{j-1/2}}^{y_{j+1/2}}\int\limits_{\theta_{k-1/2}}^{\theta_{k+1/2}}\left[\partial_x\varphi(\tilde{t},\tilde{x}, y, \theta) - \frac{\varphi(t^n,\Delta x+\tilde{s},y,\theta)-\varphi(t^n,\tilde{s},y,\theta)}{\Delta x}\right]\Delta t\Delta x\dd y\dd\theta\\
        =& \int\limits_{y_{j-1/2}}^{y_{j+1/2}}\int\limits_{\theta_{k-1/2}}^{\theta_{k+1/2}}\left[\partial_x\varphi(\tilde{t},\tilde{x}, y, \theta)- \partial_x\varphi(t^n,\hat{x},y,\theta)\right]\Delta t\Delta x\dd y\dd\theta.
    \end{split}
\end{equation}
Then, taking absolute value and using the multi-variate mean value theorem we obtain
\begin{equation}
    \label{eq:mvt2}
    \begin{split}
    &\int\limits_{y_{j-1/2}}^{y_{j+1/2}}\int\limits_{\theta_{k-1/2}}^{\theta_{k+1/2}}\left|\partial_x\varphi(\tilde{t},\tilde{x})- \partial_x\varphi(t^n,\hat{x},y,\theta)\right|\Delta t\Delta x\dd y\dd\theta\\
    \leq&\int\limits_{y_{j-1/2}}^{y_{j+1/2}}\int\limits_{\theta_{k-1/2}}^{\theta_{k+1/2}}\left(\|\partial_{tx}\varphi\|_{L^\infty(\Sigma_T)}+\|\partial_{x}^2\varphi\|_{L^\infty(\Sigma_T)}\right)(\Delta t+\Delta x)\Delta t\Delta x\dd y\dd\theta
    \\\leq& \ \|\varphi\|_{C^2(\Sigma_T)}(\Delta t+\Delta x)\Delta\xi\Delta t.
\end{split}
\end{equation}
Hence, using this estimate and the uniform bound on $\|\dd_\xi f_h\|_{L^2(\Sigma_T)}$ from \cref{cor:unifL2} in \cref{eq:dh}, we obtain $\mathcal{D}(h)-\hat{\mathcal{D}}(h)\to 0$. 

Then, for $\mathcal{P}(h)-\hat{\mathcal{P}}(h)$,  the $x$-terms are treated as follows. The $y$-term can be treated in the same manner. The $x$-terms in $\mathcal{P}(h)-\hat{\mathcal{P}}(h)$ are
\begin{equation}\label{eq:dP}
\begin{split}
    &-\int_0^T\int_{\Sigma}\Pe \cos_h(\theta) f_h\partial_x\varphi\dd\xi\dd t+\sum_{n=0}^{N_T}\sum_{i,j,k}\Pe\cos(\theta_k)f^{n+1}_{i,j,k}\dd_x\varphi^{n+1}_{i+1/2,j,k}\Delta\xi\Delta t\\
    =&-\Pe\sum_{n=1}^{N_T}\sum_{i,j,k}\left[\int\limits_{t^{n-1}}^{t^{n}}\int\limits_{C_{i+1/2,j,k}}\cos(\theta_k) f_h\partial_x\varphi\dd\xi\dd t-\cos(\theta_k)f^{n}_{i,j,k}\dd_x\varphi^{n}_{i+1/2,j,k}\Delta\xi\Delta t\right].
\end{split}
\end{equation}
We can split the integral over $C_{i+1/2,j,k}$ in two terms
\begin{equation}
\begin{split}
    I_h=&\int\limits_{C_{i+1/2,j,k}}\cos(\theta_k)f_h\partial_x\varphi\dd\xi\\=&\left[\int\limits_{C_{i,j,k}\cap C_{i+1/2,j,k}}\cos(\theta_k)f^n_{i,j,k}\partial_x\varphi\dd\xi+\int\limits_{C_{i+1/2,j,k}\cap C_{i+1,j,k}}\cos(\theta_k)f^n_{i+1,j,k}\partial_x\varphi\dd\xi\right],
\end{split}
\end{equation}
and putting it back together, we get
\begin{equation}
\begin{split}
    I_h=\int_{\theta_{k-1/2}}^{\theta_{k+1/2}}\int_{y_{j-1/2}}^{y_{j+1/2}}\cos(\theta_k)\Big[&\quad f^{n}_{i,j,k}\left(\varphi(t,x_{i+1/2},y,\theta)-\varphi(t,x_{i},y,\theta)\right)\\&+f^n_{i+1,j,k}\left(\varphi(t,x_{i+1},y,\theta)-\varphi(t,x_{i+1/2},y,\theta)\right)\Big]\dd y\dd\theta.
\end{split}
\end{equation}
Using these identities for \cref{eq:dP} and rewriting, while adding and subtracting the term 
$$
\sum_{n=1}^{N_T}\sum_{i,j,k}\int\limits_{t^{n-1}}^{t^{n}}\int\limits_{\theta_{k-1/2}}^{\theta_{k+1/2}}\int\limits_{y_{j-1/2}}^{y_{j+1/2}}\Pe\cos(\theta_k) f^n_{i,j,k}\left(\varphi(t,x_{i+1},y,\theta)-\varphi(t,x_{i+1/2},y,\theta)\right)\dd y\dd\theta \dd t,
$$
we get for the $x$-terms in $\mathcal{P}(h)-\hat{\mathcal{P}}(h)$
\begin{equation}
\begin{split}
    &\quad \sum_{n=1}^{N_T}\sum_{i,j,k}\Pe\cos(\theta_k)f^n_{i,j,k}\iiint_{C^n_{j,k}}-\left[\left(\varphi(t,x_{i+1/2},y,\theta)-\varphi(t,x_i,y,\theta)\right)-\left(\varphi^n_{i+1,j,k}-
    \varphi_{i,j,k}^n\right)\right]\dd y\dd\theta\dd t\\
    &+\sum_{n=1}^{N_T}\sum_{i,j,k}\Pe\cos(\theta_k)(-f^n_{i+1,j,k})\iiint_{C^n_{j,k}}\left(\varphi(t,x_{i+1},y,\theta)-\varphi(t,x_{i+1/2},y,\theta)\right)\dd\theta\dd y\dd t\\=&-\sum_{n=1}^{N_T}\sum_{i,j,k}\Pe\cos(\theta_k)f^n_{i,j,k}\iiint_{C^n_{j,k}}\left[\left(\varphi(t,x_{i+1},y,\theta)-\varphi(t,x_i,y,\theta)\right)-\left(\varphi^n_{i+1,j,k}-
    \varphi_{i,j,k}^n\right)\right]\dd y\dd\theta\dd t\\
    &-\sum_{n=1}^{N_T}\sum_{i,j,k}\Pe\cos(\theta_k)(f^n_{i+1,j,k}-f^n_{i,j,k})\iiint_{C^n_{j,k}}\left(\varphi(t,x_{i+1},y,\theta)-\varphi(t,x_{i+1/2},y,\theta)\right)\dd\theta\dd y\dd t\\
    =&\quad \mathcal{I}_\mathcal{P}^1+\mathcal{I}_\mathcal{P}^2,
\end{split} 
\end{equation}
where $C_{j,k}^n = (t^{n-1}, t^n) \times (y_{j-1/2}, y_{j+1/2})\times(\theta_{k-1/2}, \theta_{k+1/2})$.

The $\mathcal{I}_\mathcal{P}^1$ term converges to zero by using again the mean-value theorem and the finite-ness of the $C^2(\Sigma_T)$-norm of $\varphi$, as in \cref{eq:mvt,eq:mvt2}. 

The $\mathcal{I}_\mathcal{P}^2$ term can be upper bounded as, using the mean-value theorem and Cauchy-Schwarz,
\begin{equation}
\begin{split}
    \mathcal{I}_\mathcal{P}^2&\leq \|\varphi\|_{C^1(\Sigma_T)}\sum_{n=1}^{N_T}\sum_{i,j,k}\Pe\cos(\theta_k)(f^n_{i+1,j,k}-f^n_{i,j,k})\Delta\xi\Delta t\\
    &\leq C \|\varphi\|_{C^1(\Sigma_T)} \|\dd_x f_h\|_{L^2(\Sigma_T)} \Delta x,
\end{split}
\end{equation}
which goes to zero as $h\to 0$.

Finally, the convergence of $\mathcal{U}(h)-\hat{\mathcal{U}}(h)$ follows by a similar splitting argument as used for the $\mathcal{P}$-term. Indeed, by writing
\begin{equation}
    \begin{split}
        &\mathcal{U}(h)-\hat{\mathcal{U}}(h)\\
    =&\quad\sum_{n=1}^{N_T}\sum_{i,j,k}\gamma B[c^n]_{i,j,k+1/2}f^{n}_{i,j,k}\iiint_{C_{i,j}^n}\left[\left(\varphi(t,x,y,\theta_{k+1})-\varphi(t,x,y,\theta_k)\right)-\left(\varphi^n_{i,j,k+1}-\varphi^n_{i,j,k}
    \right)\right]\dd x\dd y\dd t\\
        &+\sum_{n=1}^{N_T}\sum_{i,j,k}\gamma B[c^n]_{i,j,k+1/2}(f^{n}_{i,j,k+1}-f_{i,j,k}^n)\iiint_{C_{i,j}^n}\left(\varphi(t,x,y,\theta_{k+1}) -\varphi(t,x,y,\theta_{k+1/2})\right)\dd x\dd y\dd t.
    \end{split}
\end{equation}
where $C_{i,j}^n = (t^{n-1}, t^n) \times (x_{i-1/2}, x_{i+1/2}) \times (y_{j-1/2}, y_{j+1/2})$. Using the uniform control on $\|B[c_h]\|_{L^2(\Sigma_T)}$, the terms also go to zero by the same argument as before. This concludes the proof.
\end{proof}

\section{Higher regularity estimates and long-time estimates}\label{sec:highreg}

In addition to the convergence result of \cref{sec:conv}, we establish a uniform-in-time estimate for the $L^2$ and $L^\infty$ norms of the discrete solution to the scheme  \eqref{eq:scheme}. For the $L^2$ norm, we can use the discrete analogue of the strategy used for the continuum problem in \cite{bruna2024lane} for $B\in\{B_0,B_\lambda, B_\tau\}$. The key is to convert the negative gradient terms into negative $L^2$ terms using the Gagliardo--Nirenberg inequality. In contrast, we require a different strategy for the $L^\infty$ norm since it is not clear what the analogue of the $W^{2,p}$-regularity theory for discrete finite volume functions is. Instead, we show a higher regularity result for the discrete $H^1$-norm of $\rho_h$, which allows us to use a discrete Morrey inequality \cite[Theorem 4.1]{porretta2020note} for the discrete gradient of $c_h$. We can cover the cases $B\in\{B_0,B_\lambda\}$.

\subsection{$L^\infty$ estimate for $\rho_h$ and the discrete Alikakos lemma}
\begin{lem}[Discrete Alikakos]\label{lem:discalik}
    Let $(F_k^n)_{n\in \mathbb N_0}$ be a family of nonnegative sequences for $k=1,2,\dots$, such that
    \begin{enumerate}
        \item[(i)] $F_1^n=1$, for all $n \in \mathbb N_0$,
        \item[(ii)] $\lim_{k\to+\infty} (F_k^0)^{1/(\lambda_k+1)}=F_\infty^0 < \infty$ exists,
        \item[(iii)] $(F^0_k)^{1/(\lambda_k+1)}\leq F^0_\infty$, for all $k=1, 2, \ldots$, 
    \end{enumerate}
    where $\lambda_k=2^k-1$. Moreover, assume that there are constants $C>0$ and $q,\beta\geq 1$, independent of $k$ such that, for $k=1,2,\dots$,
    \begin{equation}
    \label{eq:alikakos-coeffs}        
        a_k = C 2^k(2^k-1), \qquad \epsilon_k = 2^{-qk}, \qquad c_k = 2^{\beta k},
    \end{equation}
and 
    \begin{equation}
        \label{eq:discalik}
        \dfrac{1}{\Delta t}(F_k^{n+1}-F_k^{n})\leq -\epsilon_k F_k^{n+1}+(a_k+\epsilon_k)c_k\left(\sup_{n \in \mathbb N_0} F_{k-1}^n\right)^2,
    \end{equation}
   Then,  there exists a constant $A = A(\beta,q,C)\geq 1$ such that
    \begin{equation}
        \label{eq:discalik2}
        \left( \sup_{n \in \mathbb N_0} F_k^n \right)^{1/(\lambda_k+1)} \leq A K,
    \end{equation}
    for all $k=1,2,\dots$, where $K=\max\{1,F^0_\infty\}$. 
\end{lem}
\begin{proof}
    The proof is a discrete analogue of the proof of \cite[Lemma 5.1]{kowalczyk2005preventing}, based on an induction argument over $k=1,2,\ldots$. First, by assumption, we have $\sup_{n\in \mathbb N_0} F^n_1\leq K$. 
    For the induction step, let us assume that for some $M \in \mathbb N$, we have
    $$
        \left( \sup_{n \in \mathbb N_0} F^n_{k} \right)^{1/(\lambda_k+1)}\leq A_k K, \quad \text{ for } k=1,\dots,M-1,
    $$
    where $A_k > 0$ is to be determined. We note that $A_1=1$. Now, we can rewrite \cref{eq:discalik} as 
    \begin{equation}
        F^{n+1}_k\leq \frac{F^n_k+ \Delta t \epsilon_k \delta_k\left(\sup_n F_{k-1}^n\right)^2}{1+\epsilon_k\Delta t},
    \end{equation}
    where we defined $\delta_k=\frac{a_k+\epsilon_k}{\epsilon_k}c_k$. By induction on $n$, this implies that 
    \begin{equation}
        \sup_{n \in \mathbb N_0} F^n_k\leq\max\left\{F^0_k,\delta_k\left(\sup_{n \in \mathbb N_0} F_{k-1}^n\right)^2\right\}.
    \end{equation}
    By assumption (iii) and the definition of $K$, we have $F^0_k\leq K^{\lambda_k+1}$. Therefore, we can conclude that
    \begin{align}
        \label{eq:alikakos-intermediate}
        \begin{split}
        \sup_{n\in \mathbb N_0} F^n_k
        &\leq \max\{K^{\lambda_k+1},\delta_k ((A_{k-1}K)^{\lambda_{k-1}+1})^2\} \leq \max\{K^{\lambda_k+1},\delta_k A_{k-1}^{\lambda_k+1}K^{\lambda_k+1}\}\\
        &\leq \left(\max\left\{1 ,\delta_k^{1/(\lambda_k+1)} A_{k-1}\right\} K\right)^{\lambda_k+1},
        \end{split}
    \end{align} 
    having used the definition of $\lambda_k$. Therefore, it follows that, for $k\geq 2$, 
    \begin{equation}\label{eq:Ak}
        A_k^{\lambda_k+1} = \delta_k \delta_{k-1}^2\dots(\delta_2)^{2^{k-2}}=\prod_{\ell=2}^k \delta_\ell^{2^{k-\ell}}.
    \end{equation}
    By the definition of $\delta_k$ and \cref{eq:alikakos-coeffs}, we also note that 
    $$
        \delta_k \leq (1+C)2^{(2+q+\beta)k},
    $$
    such that we can write for \cref{eq:Ak}
    $$
        A_k^{\lambda_k+1} \leq \prod_{\ell=2}^k(1+C)^{2^{k-\ell}}2^{(2+q+\beta)\ell 2^{k-\ell}} = \prod_{\ell=0}^{k-2} (1+C)^{2^{\ell}}2^{(2+q+\beta)(k-\ell) 2^{\ell}},
    $$ 
    such that
    \cref{eq:alikakos-intermediate} becomes,
    \begin{equation}
        \sup_{n \in \mathbb N_0} F^n_k\leq K^{\lambda_k+1}\prod_{\ell=0}^{k-2}(1+C)^{2^{\ell}}2^{(2+q+\beta)(k-\ell)2^{\ell}}.
    \end{equation}
    We now use the identities 
    $$
        \sum_{\ell=0}^{k-2} 2^\ell = 2^{k-1}-1, \qquad \text{and} \qquad \sum_{\ell=0}^{k-2} \ell 2^\ell = 2 + 2^{k-1}(k-3),
    $$
    to conclude that 
    \begin{equation}
    \begin{split}
        \sup_{n\in \mathbb N_0} F^n_k
        &\leq K^{2^k}(1+C)^{2^{k-1}-1}2^{(2+q+\beta)(3\cdot 2^{k-1}-k-2)}\leq K^{2^k}(1+C)^{2^k}2^{(2+q+\beta)2^{k+1}}.
        \end{split}
    \end{equation}
    Therefore 
    \begin{equation}
        \left( \sup_{n \in \mathbb N_0} F^n_k \right)^{1/(\lambda_k+1)}\leq (1+C) 2^{2(2+q+\beta)}K.
    \end{equation}
   Hence, we have obtained the inequality of \cref{eq:discalik2} with $A = ( 1 + C ) 2^{2(2+q+\lambda)}$.
\end{proof}
\begin{cor}
    \label{cor:1}
    Assume the initial data satisfies $\|\rho^0\|_{L^\infty}<+\infty$. Then, the piecewise-constant approximation of the spatial density of the solution of the numerical scheme, \cref{eq:scheme-rho}, satisfies
    \begin{equation}
        \sup_{t\geq 0} \|\rho_h(t)\|_{L^{\lambda_k+1}}\leq A(\beta,q)\max\{1,\|\rho^0\|_{L^\infty}\},
    \end{equation}
    for $\lambda_k=2^k-1$ and $k=1,2,\dots$, and for some $C,\beta,q>0$, independent of the mesh size.
\end{cor}
\begin{proof}
    The idea of the proof is a discrete analogue of \cite[Proposition 3.5]{bruna2024lane}. We first multiply \cref{eq:scheme-rho} by $(p+1)(\rho_{i,j}^{n+1})^p\Delta\xx$ and sum over all $(i,j)$ to obtain
    \begin{equation}
    \label{eq:rhop1}
        \begin{split}
            (p+1)
            &\sum_{i,j}\frac{\Delta\xx}{\Delta t}(\rho^{n+1}_{i,j}-\rho^n_{i,j})(\rho^{n+1}_{i,j})^p\\
            &=(p+1)\sum_{i,j}D_T\Delta\xx \left(\dd_x^2\rho^{n+1}_{i,j} +\dd_y^2\rho^{n+1}_{i,j}\right)\left(\rho^{n+1}_{i,j}\right)^p\\
            &\qquad -(p+1)\sum_{i,j}\Pe\Delta\xx\left[\dd_x\bar{U}^{n+1}_{i,j}(\rho^{n+1}_{i,j})^p+\dd_y\bar{V}_{i,j}^{n+1}(\rho^{n+1}_{i,j})^p\right]\\
            &=:I_{D_T}+I_{\Pe}.
        \end{split}
    \end{equation}
    By summation by parts, we have 
    \begin{equation}
        I_{\Pe}=\Pe(p+1)\sum_{i,j}\Delta\xx\left[\bar{U}^{n+1}_{i+1/2,j}\dd_x(\rho^{n+1}_{i+1/2,j})^p+\bar{V}^{n+1}_{i,j+1/2}\dd_y(\rho^{n+1}_{i,j+1/2})^p\right].
    \end{equation}
    Using the elementary inequality 
    $$
        |b-a|\leq \frac{2p}{p+1} \max\left\{b^{\frac{p-1}{2p}},a^{\frac{p-1}{2p}}\right\} \left|b^{\frac{p+1}{2p}}-a^{\frac{p+1}{2p}}\right|,
    $$ 
    for $a,b\in\mathbb{R}_{\geq 0}$, with $b=(\rho^{n+1}_{i+1,j})^p,a=(\rho^{n+1}_{i,j})^p$ in conjunction with  $|\bar{U}^{n+1}_{i+1/2,j}|\leq\max\{\rho^{n+1}_{i,j},\rho^{n+1}_{i+1,j}\}$, we obtain
    $$
        \left|\bar{U}^{n+1}_{i+1/2,j}\dd_x(\rho^{n+1}_{i+1/2,j})^p\right|\leq\frac{2p}{p+1} \max\{\rho^{n+1}_{i,j},\rho^{n+1}_{i+1,j}\}^{\frac{p+1}{2}}\left|\dd_x (\rho^{n+1}_{i+1/2,j})^{\frac{p+1}{2}}\right|.
    $$
    Similarly, we get
    $$
        \left|\bar{V}^{n+1}_{i,j+1/2}\dd_y(\rho^{n+1}_{i,j+1/2})^p\right|\leq\frac{2p}{p+1} \max\{\rho^{n+1}_{i,j},\rho^{n+1}_{i,j+1}\}^{\frac{p+1}{2}}\left|\dd_y (\rho^{n+1}_{i,j+1/2})^{\frac{p+1}{2}}\right|.
    $$
    Returning to \cref{eq:rhop1}, we get for the $x$-part of $I_\Pe$,
    denoted by $I_{\Pe,x}$,
    \begin{equation}
    \begin{split}
        I_{\Pe,x}\leq& \Pe(p+1)\sum_{i,j}\frac{2p}{p+1} \max\{\rho^{n+1}_{i,j},\rho^{n+1}_{i+1,j}\}^{\frac{p+1}{2}}|\dd_x (\rho^{n+1}_{i+1/2,j})^{\frac{p+1}{2}}|\Delta\xx.
    \end{split}
    \end{equation}
    Now, using Young--H\"older we have
\begin{multline*}
	        \max\left\{\rho^{n+1}_{i,j},\rho^{n+1}_{i+1,j}\right\}^{\frac{p+1}{2}}\frac{2}{p+1} \left|\dd_x (\rho^{n+1}_{i+1/2,j})^{\frac{p+1}{2}}\right| \\ \leq \frac{1}{2\varepsilon} \max\left\{\rho^{n+1}_{i,j},\rho^{n+1}_{i+1,j}\right\}^{p+1} + \frac{4\Pe^2}{(p+1)^2}\frac{\varepsilon}{2} \left|\dd_x (\rho^{n+1}_{i+1/2,j})^{\frac{p+1}{2}}\right|^2.
\end{multline*}
    Applying the above inequalities analogously to $I_{\Pe,y}$, and using $(p+1)b^p(b-a)\geq (b^{p+1}-a^{p+1})$ for the time derivative,
    \begin{equation}
    \begin{split}
        \dfrac{1}{\Delta t}
        &\left( \|\rho_h(t^{n+1})\|_{L^{p+1}}^{p+1}-\|\rho_h(t^{n})\|_{L^{p+1}}^{p+1} \right)\\
        &\leq -(p+1)D_T\sum_{i,j}\Delta\xx \left (\dd_x\rho^{n+1}_{i+1/2,j}\dd_x(\rho^{n+1}_{i+1/2,j})^p+\dd_y\rho^{n+1}_{i,j+1/2}\dd_y(\rho^{n+1}_{i,j+1/2})^p\right )\\
        &\quad +p(p+1)\sum_{i,j}\Delta\xx\left (\tfrac{1}{2\varepsilon}\max\{\rho^{n+1}_{i,j},\rho^{n+1}_{i+1,j}\}^{p+1}+\frac{4\Pe^2\varepsilon}{2(p+1)^2}\left|\dd_x (\rho^{n+1}_{i+1/2,j})^{\frac{p+1}{2}}\right|^2\right) \\
        &\quad +p(p+1)\sum_{i,j}\Delta\xx\left(\tfrac{1}{2\varepsilon}\max\{\rho^{n+1}_{i,j},\rho^{n+1}_{i,j+1}\}^{p+1}+\frac{4\Pe^2\varepsilon}{2(p+1)^2}\left|\dd_y (\rho^{n+1}_{i,j+1/2})^{\frac{p+1}{2}}\right|^2\right).
    \end{split}
    \end{equation}
    Using the elementary inequality $-(p+1)(a-b)(a^p-b^p)\leq-\frac{4p}{p+1}(a^{\frac{p+1}{2}}-b^{\frac{p+1}{2}})^2$, the first sum can be bounded as
    \begin{equation}
        \label{ineq:p12}
        I_{D_T}
        \leq -\frac{4p D_T}{p+1}
        \left( \left\|\dd_x(\rho_h(t^{n+1}))^{\frac{p+1}{2}}\right\|_{L^2}^2 + \left\|\dd_y(\rho_h(t^{n+1}))^{\frac{p+1}{2}}\right\|_{L^2}^2\right).
    \end{equation}
    Combining the above and noting $\max\{a,b\}^p\leq a^p+b^p$ for $a,b\in\mathbb{R}_{\geq 0},p\geq 1$, we can write
        \begin{multline}
            \frac{1}{\Delta t}
            \left(\|\rho_h(t^{n+1})\|_{L^{p+1}}^{p+1}-\|\rho_h(t^{n})\|_{L^{p+1}}^{p+1}\right)\\
            \leq -\frac{4p}{p+1}(D_T-\tfrac{\varepsilon}{2}\Pe^2)\|\dd_\xx(\rho_h)^{\frac{p+1}{2}}\|_{L^2}^2+\tfrac{2}{\varepsilon}p(p+1)\|\rho_h(t^{n+1})\|_{L^{p+1}}^{p+1}.
                \end{multline}
    We now use the discrete Gagliardo--Nirenberg inequality \cite[Theorem 3.4]{bessemoulin2015discrete} so that, for $u=\rho_h^{\frac{p+1}{2}}$,
    \begin{equation}
        \|u\|_{L^2}^2\leq C_{GN}\|u\|_{L^1}\|u\|_{1,2}.
    \end{equation}
    After some algebra (see \cref{app:alikalgebra}) for constants $a_k,\epsilon_k$ and $c_k$ as in \cref{lem:discalik}, we obtain
    \begin{equation}\label{eq:AlikRhoAlgebra}
    \begin{split}
        \frac{1}{\Delta t} \left(\|\rho_h(t^{n+1})\|_{L^{p+1}}^{p+1}-\|\rho_h(t^{n})\|_{L^{p+1}}^{p+1}\right) \leq-\epsilon_k\|\rho_h(t^{n+1})\|_{L^{p+1}}^{p+1}+(a_k+\epsilon_k) c_k\left(\|\rho_h(t^{n+1})\|_{L^{\frac{p+1}{2}}}^{\frac{p+1}{2}}\right)^2.
    \end{split}
    \end{equation}
    Therefore, by inserting the notation $\lambda_k=p=2^k-1$ we get the result by \cref{lem:discalik}.
\end{proof}
\subsection{$L^\infty$ estimate for $f_h$}
\begin{prop}[Uniform $L^2$-estimate for $f_h$]
\label{prop:L2unif}
    Assume the conditions in \cref{cor:1} are met. Then, any solution to the numerical scheme as in \cref{eq:scheme} satisfies
    \begin{equation}
        \sup_{t\geq 0}\|f_h(t)\|_{L^2}< C,
    \end{equation}
    for some constant $C>0$ that does not depend on the mesh size.
\end{prop}
\begin{proof}
    The idea is to revisit \cref{eq:intermL2-3}, i.e.,
        \begin{equation}
    \label{ineq:GNfL2}
        \begin{split}
            \frac{1}{2\Delta t} \left(\|f_h(t^{n+1})\|_{L^2}^2-\|f_h(t^{n})\|_{L^2}^2 \right)\leq-D_T\|\dd_\xx f_h(t^{n+1})\|_{L^2}^2-\|\dd_\theta f_h(t^{n+1})\|_{L^2}^2\\
            +\frac{\gamma}{2}\int |\dd_\theta B[c_h(t^{n+1})]|\, |f_h(t^{n+1})|^2\dd\xi,
        \end{split}
    \end{equation}
    and use an improved version of the estimate of \cref{eq:BcL2} to estimate the last term. Applying the discrete Gagliardo--Nirenberg inequality \cite[Theorem 3.4]{bessemoulin2015discrete} 
    \begin{equation}
        -\|\dd_\xi f_h\|_{L^2}^2\leq -\|f_h\|_{L^2} + C \|f_h\|_{L^1}^{\frac{4}{7}} = C -\|f_h\|_{L^2},
    \end{equation}
    for some constant $C>0$, to the first two terms, \cref{ineq:GNfL2} becomes
    \begin{equation}
    \label{ineq:GNfL2-smplfd}
        \begin{split}
            \frac{1}{2\Delta t}(\|f_h(t^{n+1})\|_{L^2}^2-\|f_h(t^{n})\|_{L^2}^2)
            &\leq C - \kappa \|f_h(t^{n+1})\|_{1,2}^2\\
            &\quad +\frac{\gamma}{2}\int |\dd_\theta B[c_h(t^{n+1})]|\, |f_h(t^{n+1})|^2\dd\xi,
        \end{split}
    \end{equation}
    where $\kappa ={\min(1, D_T)}/{2}$. Next, using the H\"older inequality and then an interpolation of Lebesgue spaces, we observe
    \begin{align}
    \begin{aligned}
        \frac{\gamma}{2}\int|\dd_\theta B[c_h]|\, |f_h(t^{n+1})|^2\dd\xi
        &\leq C\|\dd_\theta B[c_h]\|_{L^2}\|f_h(t^{n+1})\|_{L^4}^2\\ 
        &\leq C \|f_h(t^{n+1})\|_{L^1}^{\frac{2}{10}}\|f_h(t^{n+1})\|_{L^6}^{\frac{18}{10}} \\
        &\leq C \|f_h(t^{n+1})\|_{L^6}^{\frac{9}{5}},
            \end{aligned}
    \end{align}
    having also used mass conservation, \cref{cor:disccreg} and \cref{cor:1}. Here, the constant $C$ changes from line to line but remains independent of the mesh size.
    Then, \cref{ineq:GNfL2-smplfd} becomes
    \begin{equation}
    \label{ineq:GNfL2-smplfd1}
        \begin{split}
            \frac{1}{2\Delta t}(\|f_h(t^{n+1})\|_{L^2}^2-\|f_h(t^{n})\|_{L^2}^2)
            &\leq C - \kappa \|f_h(t^{n+1})\|_{1,2}^2 + C \|f_h\|_{L^6}^{\frac{9}{5}}.
        \end{split}
    \end{equation}
    Next, we can apply the discrete Poincar\'e--Sobolev inequality 
    \begin{align}
        \|f_h\|_{L^6} \leq C_{PS} \|f_h\|_{1,2}
    \end{align}
    cf. \cite[Theorem 3.2]{bessemoulin2015discrete}, to estimate the $L^6$-norm in \cref{ineq:GNfL2-smplfd1}. This way, we obtain
    \begin{equation}
        \frac{1}{\Delta t}(\|f_h(t^{n+1})\|_{L^2}^2-\|f_h(t^{n})\|_{L^2}^2)\leq -\frac\kappa2\|f_h(t^{n+1})\|_{L^2}^2 + C(\kappa),
    \end{equation}
    for positive constant $C(\kappa)>0$ that does not depend on the mesh size. Hence, we find 
    \begin{equation}
        \sup_{n\in \mathbb N_0} \|f_h(t^{n})\|_{L^2}^2\leq C,
    \end{equation}
    where $C>0$ only depends on $\alpha$ and $\|f_0\|_{L^2}$ but is independent of the mesh size. Then, the result follows.
\end{proof}
\begin{prop}[$L^2$-estimate for $\dd_\xx\rho_h$]\label{prop:rhoh1}
Given  $\|\nabla_\xx \rho^0\|_{L^2(\Omega)}<+\infty$ and the same assumptions of \cref{prop:L2unif}, the solution to the numerical scheme \eqref{eq:scheme} satisfies
\begin{equation}
    \sup_{t\geq 0}\|\dd_\xx\rho_h(t)\|_{L^2(\Omega)}\leq C,
\end{equation}
where $C>0$ is independent of the mesh size.
\end{prop}
\begin{proof}
    The idea is to obtain a discrete analogue of the regularity result for linear parabolic equations in \cite[Section 7.1.3]{evans2010partial}. 
    Recalling \cref{eq:scheme-rho}, we know that $\rho^n_{i,j}$ satisfies
    \begin{equation}
        \frac{\rho^{n+1}_{i,j}-\rho^n_{i,j}}{\Delta t}
        =
        D_T(\dd^2_x \rho^{n+1}_{i,j} + \dd^2_y \rho^{n+1}_{i,j}) -\dd_x \bar U_{i,j}^{n+1} - \dd_y \bar V_{i,j}^{n+1},
    \end{equation}
    where the $\bar{U}$ and $\bar{V}$ terms are as before, cf. \cref{eq:fluxes-rho}.
    
    Next, we move the discrete Laplacian on the left-hand side, square both sides, multiply by $\Delta\xi$, and sum over all $i,j$. This yields
    \begin{multline}
        \label{eq:ddtgradrho}
             \sum_{i,j} \Delta\xi \left(\left|\frac{\rho^{n+1}_{i,j}-\rho^n_{i,j}}{\Delta t}\right|^2 - 2 D_T \frac{\rho^{n+1}_{i,j}-\rho^n_{i,j}}{\Delta t} (\dd^2_x \rho^{n+1}_{i,j} + \dd^2_y \rho^{n+1}_{i,j}) + D_T^2 \left|\dd^2_x\rho^{n+1}_{i,j} + \dd^2_y\rho^{n+1}_{i,j}\right|^2\right)\\
            = \sum_{i,j} \Delta\xi \left(\left|\dd_x\bar{U}^{n+1}_{i,j}\right|^2 + 2 \dd_x\bar{U}^{n+1}_{i,j}\dd_y\bar{V}^{n+1}_{i,j} + \left|\dd_y\bar{V}^{n+1}_{i,j}\right|^2\right).
    \end{multline}
    We observe that all but one term on the left-hand side are squares. Concerning the cross term, we obtain by a summation by parts
    \begin{align}
        \label{eq:ddtgradrho-interm}
        \begin{aligned}
        - 2 D_T  &\sum_{i,j} \Delta \xi \frac{\rho^{n+1}_{i,j}-\rho^n_{i,j}}{\Delta t} (\dd^2_x \rho^{n+1}_{i,j} + \dd^2_y \rho^{n+1}_{i,j})\\
        &=2 D_T \sum_{i,j} \Delta\xi \left(\frac{\dd_x\rho^{n+1}_{i+1/2,j}-\dd_x\rho^{n}_{i+1/2,j}}{\Delta t}\dd_x\rho^{n+1}_{i+1/2,j} + \frac{\dd_y\rho^{n+1}_{i,j+1/2}-\dd_y\rho^{n}_{i,j+1/2}}{\Delta t}\dd_y\rho^{n+1}_{i,j+1/2}\right)\\
        &\geq D_T \sum_{i,j} \Delta\xi  \left(\frac{\left|\dd_x\rho^{n+1}_{i+1/2,j}\right|^2-\left|\dd_x\rho^{n}_{i+1/2,j}\right|^2}{\Delta t}+\frac{ \left|\dd_y\rho^{n+1}_{i,j+1/2}\right|^2-\left|\dd_y\rho^{n}_{i,j+1/2}\right|^2}{\Delta t}\right),
        \end{aligned}
        \end{align}
    where we used the inequality $(a-b)a\geq\frac{1}{2}(a^2-b^2)$ in the last line. Substituting \cref{eq:ddtgradrho-interm} into \cref{eq:ddtgradrho}, we obtain
    \begin{align}
            \begin{aligned}
        D_T \sum_{i,j} \Delta\xi  & \left(\frac{\left|\dd_x\rho^{n+1}_{i+1/2,j}\right|^2-\left|\dd_x\rho^{n}_{i+1/2,j}\right|^2}{\Delta t}+\frac{ \left|\dd_y\rho^{n+1}_{i,j+1/2}\right|^2-\left|\dd_y\rho^{n}_{i,j+1/2}\right|^2}{\Delta t}\right) \\
        &\leq \sum_{i,j} \Delta\xi \left(\left|\dd_x\bar{U}^{n+1}_{i,j}\right|^2 + 2 \dd_x\bar{U}^{n+1}_{i,j}\dd_y\bar{V}^{n+1}_{i,j} + \left|\dd_y\bar{V}^{n+1}_{i,j}\right|^2\right)\\
        & \leq C \sum_{i,j,k} \Delta \xi |\dd_\xx f^{n+1}_{i,j,k}|^2,
        \end{aligned}
    \end{align}
    where we used Jensen's inequality and  Cauchy--Schwarz, to estimate  $|\dd_x \bar{U}^{n+1}_{i,j}|\leq |\dd_x f^{n+1}_{i-1/2,j}|+|\dd_x f^{n+1}_{i+1/2,j}|$ and $\dd_y\bar{V}$, respectively. In summary, we have, for some $C>0$, 
    \begin{equation}\label{eq:DTC}
        \sum_{i,j} \Delta\xi \frac{|\dd_\xx\rho_{i,j}^{n+1}|^2-|\dd_\xx\rho_{i,j}^{n}|^2}{\Delta t}\leq C \sum_{i,j,k} \Delta \xi |\dd_\xx f^{n+1}_{i,j,k}|^2.
    \end{equation}
    Fix $t>0$ and let $m=\lfloor t/\Delta t\rfloor$, $M=\lfloor (t+1)/\Delta t\rfloor$ and let $m\leq s\leq M$. Then, multiplying \eqref{eq:DTC} by $\Delta t$ and summing $n$ from $s$ to $M$, we get
    \begin{equation}\label{ineq:rhodxf}
    \begin{split}
        \|\dd_\xx\rho(t^{m+1})\|_{L^2(\Omega)}^2 & \leq\|\dd_\xx\rho(t^{s})\|_{L^2(\Omega)}^2+C\sum_{n=s}^{M}\Delta t\|\dd_\xx f(t^{n+1})\|_{L^2(\Omega)}^2\\& \leq \|\dd_\xx\rho(t^{s})\|_{L^2(\Omega)}^2+C\sum_{n=m}^{M}\Delta t\|\dd_\xx f(t^{n+1})\|_{L^2(\Omega)}^2.
    \end{split}
    \end{equation}
    We note that by Inequality \eqref{ineq:L2pregronwall} in the proof of \cref{cor:unifL2}, we have 
    \begin{equation}
    \begin{split}\label{eq:DTC2}
        \|\dd_\xi f_h(t^{n+1})\|_{L^2}^2\leq& \frac{C(\varepsilon)}{D(\varepsilon)}\|\dd_\theta B[c_h(t^{n+1})]\|_{L^2}^2\|f_h(t^{n+1})\|_{L^2}^2+\frac{1}{2\Delta t D(\varepsilon)}\left(\|f_h(t^{n})\|_{L^2}^2-\|f_h(t^{n+1})\|_{L^2}^2\right)\\
        \leq& C\left(\|f_h(t^{n+1})\|_{L^2}^2+\frac{1}{\Delta t }\left(\|f_h(t^{n})\|_{L^2}^2-\|f_h(t^{n+1})\|_{L^2}^2\right)\right),
    \end{split}
    \end{equation}
    for some $C>0$. Here, we used the uniform bound $\sup_{t\geq 0}\|\dd_\theta B[c_h](t)\|_{L^2}\leq C$ for some $C>0$, which follows from \cref{cor:disccreg} and \cref{prop:L2unif}. Then, multiplying \eqref{eq:DTC2} by $\Delta t$ and summing from $n=m$ to $M$, we have, for some $C>0$ independent of the time and mesh size,
    \begin{equation}\label{eq:DTC3}
        \sum_{n=m}^M\Delta t\|\dd_\xx f_h(t^{n+1})\|_{L^2}^2\leq C\left(\sum_{n=m}^M\|f_h(t^{n+1})\|_{L^2}^2\Delta t+\|f_h(t^{m})\|_{L^2}^2\right)\leq C.
    \end{equation}
    Here, we made use of the uniform bound $\sup_{t\geq 0}\|f_h(t)\|_{L^2}^2$ from \cref{prop:L2unif} and the inequality $(M+1-m)\Delta t\leq 2$ (which we have by construction and for $\Delta t<1$).

    Hence, multiplying \eqref{ineq:rhodxf} by $\Delta t$ and summing $s$ from $m$ to $M$,
    we get, using $\frac{1}{2}\leq\sum_{n=m}^M\Delta t\leq 2$, for some $C>0$
    \begin{equation}
    \label{ineq:rhorhoc}
        \|\dd_\xx\rho(t^{m+1})\|_{L^2(\Omega)}^2\leq C\left( 1+ \sum_{n=m}^M\Delta t\|\dd_\xx\rho(t^n)\|_{L^2(\Omega)}^2\right).
    \end{equation}
    Finally, using $\|\dd_\xx\rho(t^n)\|_{L^2(\Omega)}\leq C\|\dd_\xx f(t^n)\|_{L^2}$, for some constant $C>0$,
    we can bound the right hand side of \eqref{ineq:rhorhoc} via \eqref{eq:DTC3}. So, we may conclude, for some $C>0$ independent of $m$ and the mesh size,
    \begin{equation}
        \|\dd_\xx\rho(t^{m+1})\|_{L^2(\Omega)}^2\leq C.
    \end{equation}
    Hence, the result follows for any $t\geq 0$.
\end{proof}
\begin{prop}[$L^\infty$-estimate on $\dd_\xx c_h$]\label{prop:Linftygradc}
    With the same assumptions as in \cref{cor:1} and \cref{prop:rhoh1}, the solution to the numerical scheme as in \cref{eq:scheme} satisfies
    \begin{equation}
        \sup_{t\geq 0}\|\dd_\xx c_h(t)\|_{L^\infty(\Omega)}\leq C,
    \end{equation}
    for a constant $C>0$ that does not depend on the mesh size.
\end{prop}
\begin{proof}
    We first repeat the same elliptic $L^2$-estimate as in \cref{eq:L2c} for the gradient of $c_h$, by applying $\dd_x$ and $\dd_y$ to \eqref{eq:discc}. First, applying $\dd_x$ to \eqref{eq:discc} gives
        \begin{equation}\label{eq:ddxxxc}
            -\dd_x(\dd_x^2 c+\dd_y^2 c)_{i+1/2,j}=-\alpha\dd_x c_{i+1/2,j}+\dd_x\rho_{i+1/2,j},
        \end{equation}
    such that 
    \begin{equation}
    \begin{split}
        \dd_x\dd_x^2 c_{i+1/2,j}&=\frac{1}{\Delta x}
    \left(\dd^2_x c_{i+1,j}-\dd^2_x c_{i,j}\right)\\&=\frac{1}{\Delta x^3}\left(c_{i+2,j}-2c_{i+1,j}+c_{i,j}-c_{i+1,j}+2c_{i,j}-c_{i-1,j}\right)
    \\&=\frac{1}{\Delta x^2}\left(\dd_x c_{i+3/2,j}-2\dd_x c_{i+1/2,j}+\dd_x c_{i-1/2,j}\right)\\
    &=\dd_x^2 \dd_x c_{i+1/2,j}.
    \end{split}
    \end{equation}
    Squaring both sides of \eqref{eq:ddxxxc}, multiplying by $\Delta\xx$ and summing over $(i,j)$, then leads to
    \begin{equation}
    \label{eq:dxdxc}
    \begin{split}
        \sum_{i,j}\left|\dd_x(\dd_x^2 c+\dd_y^2 c)_{i+1/2,j}\right|^2\Delta\xx &\leq
        C \sum_{i,j}[\alpha^2 |\dd_x c_{i+1/2,j}|^2 + |\dd_x\rho_{i+1/2,j}|^2]\Delta\xx\\
        &\leq C (\|\dd_\xx c_h\|_{L^2(\Omega)}^2+\|\dd_\xx\rho_h\|_{L^2(\Omega)})\\
        &\leq C,
    \end{split}
    \end{equation}
    by \cref{prop:rhoh1} and \cref{cor:disccreg}, 
    for some $C>0$ independent of the mesh size. Similarly, we obtain
    \begin{equation}
        \sum_{i,j}\left|\dd_y(\dd_x^2 c+\dd_y^2 c)_{i,j+1/2}\right|^2\Delta\xx 
        \leq C.
    \end{equation}
    Expanding the left-hand side of \cref{eq:dxdxc}, an integration by parts yields
    \begin{subequations}
        \label{eq:deriv-grad-c}
    \begin{equation}
        \sum_{i,j} \left( |\dd_x\dd_x^2 c_{i+1/2,j}|^2 + 2|\dd_x\dd_{xy} c_{i,j+1/2}|^2+|\dd_x\dd_y^2 c_{i+1/2,j}|^2 \right) \Delta\xx
        \leq C,
    \end{equation}
    and similarly 
     \begin{equation}
        \sum_{i,j} \left(|\dd_y\dd_x^2 c_{i,j+1/2}|^2 + 2|\dd_y\dd_{xy} c_{i+1/2,j}|^2+|\dd_y\dd_y^2 c_{i,j+1/2}|^2 \right)\Delta\xx
        \leq C.
    \end{equation}
    \end{subequations}
    This implies that we have componentwise control over the $\|\cdot\|_{1,2}$-norm of the Hessian $\dd_\xx^2 c_h$. Now, because of the discrete Poincar\'e--Sobolev inequality \cite[Theorem 3.2]{bessemoulin2015discrete}, we have, in dimension two,
    \begin{align}
    \begin{split}
       \| \dd_x^2 &c_h\|_{L^3(\Omega)} + 
        \|\dd_{xy} c_h\|_{L^3(\Omega)} +
        \|\dd_y^2 c_h\|_{L^3(\Omega)}\\
       &\leq C_{PS}
       (\|\dd_x^2 c_h\|_{1,2,\Omega}
       +\|\dd_{xy} c_h\|_{1,2,\Omega}
       +
       \|\dd_y^2 c_h\|_{1,2, \Omega})\\
       &\leq C,
    \end{split}
    \end{align}
    where we used \cref{eq:deriv-grad-c}, that is
    $$
    |\dd_\xx c_h|_{1,3, \Omega} < C,
    $$
    for some constant $C>0$ independent of the mesh size. Therefore, by the discrete Morrey inequality \cite[Theorem 4.1]{porretta2020note}
    \begin{equation}
        \|\dd_\xx c_h\|_{L^\infty(\Omega)}\leq C_M|\dd_\xx c_h|_{1,3}^{\frac{3}{4}} \|\dd_\xx c_h\|_{L^2(\Omega)}^{\frac{1}{4}}.
    \end{equation}
    Now, the right-hand side can be uniformly bounded using  \cref{eq:deriv-grad-c} as well as \cref{prop:rhoh1}, which concludes the proof.
\end{proof}

\begin{cor}[$L^\infty$-estimate for $f_h$]
    Provided the same assumptions as for \cref{prop:Linftygradc} hold, the solution to the numerical scheme as in \cref{eq:scheme}, for the interaction terms $B\in\{B_0,B_\lambda\}$, satisfies 
    \begin{equation}
        \sup_{t\geq 0}\|f_h(t)\|_{L^\infty}\leq C,
    \end{equation}
    for some constant $C$ independent of the mesh size.
\end{cor}
\begin{proof}
    The idea of the proof is a discrete analogue of \cite[Proposition 3.5]{bruna2024lane}. We multiply \cref{eq:scheme} by $(p+1)(f_{i,j,k}^{n+1})^{p}$. After summation by parts and using Young's product inequality, we get, for some $\varepsilon>0$, 
    \begin{equation}
    \begin{split}
        \frac{1}{\Delta t}\bigg(\|&f_h(t^{n+1})\|_{L^{p+1}}^{p+1}  -\|f_h(t^{n})\|_{L^{p+1}}^{p+1}\bigg)\\ &\leq
        -\frac{4p}{p+1}\left(\min\{D_T,1\}-\gamma^2 C_c \frac{\varepsilon}{2}\right)\left\|\dd_\xi[f_h(t^{n+1})^{\frac{p+1}{2}}]\right\|_{L^2}^2+\frac{1}{\varepsilon}p(p+1)\|f_h(t^{n+1})\|_{L^{p+1}}^{p+1}
        \\ &\leq
        -\frac{4p}{p+1}C\left\|\dd_\xi[f_h(t^{n+1})^{\frac{p+1}{2}}]\right\|_{L^2}^2+\frac{1}{\varepsilon}p(p+1)\|f_h(t^{n+1})\|_{L^{p+1}}^{p+1},
    \end{split}
    \end{equation}
    where $C_c$ comes from \cref{prop:Linftygradc}, and $C>0$ is some constant. Then, using the discrete Gagliardo--Nirenberg inequality \cite[Theorem 3.4]{bessemoulin2015discrete} for dimension three
    \begin{equation}
        \|u\|_{L^2}\leq C_{GN}\|u\|_{L^1}^{\frac{2}{5}}\|u\|_{1,2}^{\frac{3}{5}},
    \end{equation}
    we can estimate
    \begin{equation}\label{eq:AlikFAlgebra}
        \begin{split}
            &\frac{1}{\Delta t}(\|f_h(t^{n+1})\|_{L^{p+1}}^{p+1}-\|f_h(t^{n})\|_{L^{p+1}}^{p+1})\leq-\epsilon_k\|f_h(t^{n+1})\|_{L^{p+1}}^{p+1}+(a_k+\epsilon_k)c_k\left(\|f_h(t^{n+1})\|_{L^{\frac{p+1}{2}}}^{\frac{p+1}{2}}\right)^2,
        \end{split}
    \end{equation}
    where $\epsilon_k,a_k,c_k$ are as in \cref{lem:discalik} (see the appendix for the algebra). Hence, the result follows by an application of \cref{lem:discalik}.
\end{proof}

  \section{Numerical results} \label{sec:numerics}

In this section, we test the numerical scheme analysed in the previous sections. Since we do not have access to explicit solutions, we test the convergence of the numerical scheme \eqref{eq:scheme} by comparing solutions with increasingly finer meshes. For ease of plotting, we choose $y$-invariant initial data so that, given the form of the equations and boundary conditions of System \eqref{full_problem_intro}, solutions remain $y$-invariant for $t>0$, and we may use a $y$-invariant version of the numerical scheme (corresponding to setting $N_y = 1$ in \eqref{def:scheme}). 
An example of a fully three-dimensional numerical solution is given in \cref{fig:btau}, in \cref{app:fig3d}, confirming the behaviour found in \cite{bertucci2024curvature,bruna2024lane}. The numerical scheme is implemented in \texttt{Julia} and can be accessed in \cite{odewit8github}. 

First, we consider the time evolution of the spatial density $\rho(t,x)$ with the interaction $B_0$, \cref{eq:B0}, with $N_x = N_\theta = N$. \cref{fig:rhoev} shows the resulting numerical solutions for three different mesh sizes $N = 32, 64, 128$ up to the final time $T = 1.0$, with the characteristic aggregation behaviour as in \cite{bruna2024lane}. 
The solution is close to a steady state at the final time $T$, as indicated by $L^2$- and $L^\infty$-norms of the discrete time derivative being of order $10^{-3}$ (not shown here). 
We also observe that solutions stay uniformly bounded for all times, consistent with our analysis in \cref{sec:highreg}.

\begin{figure}[bt]
    \centering
    \includegraphics[width=0.95\textwidth]{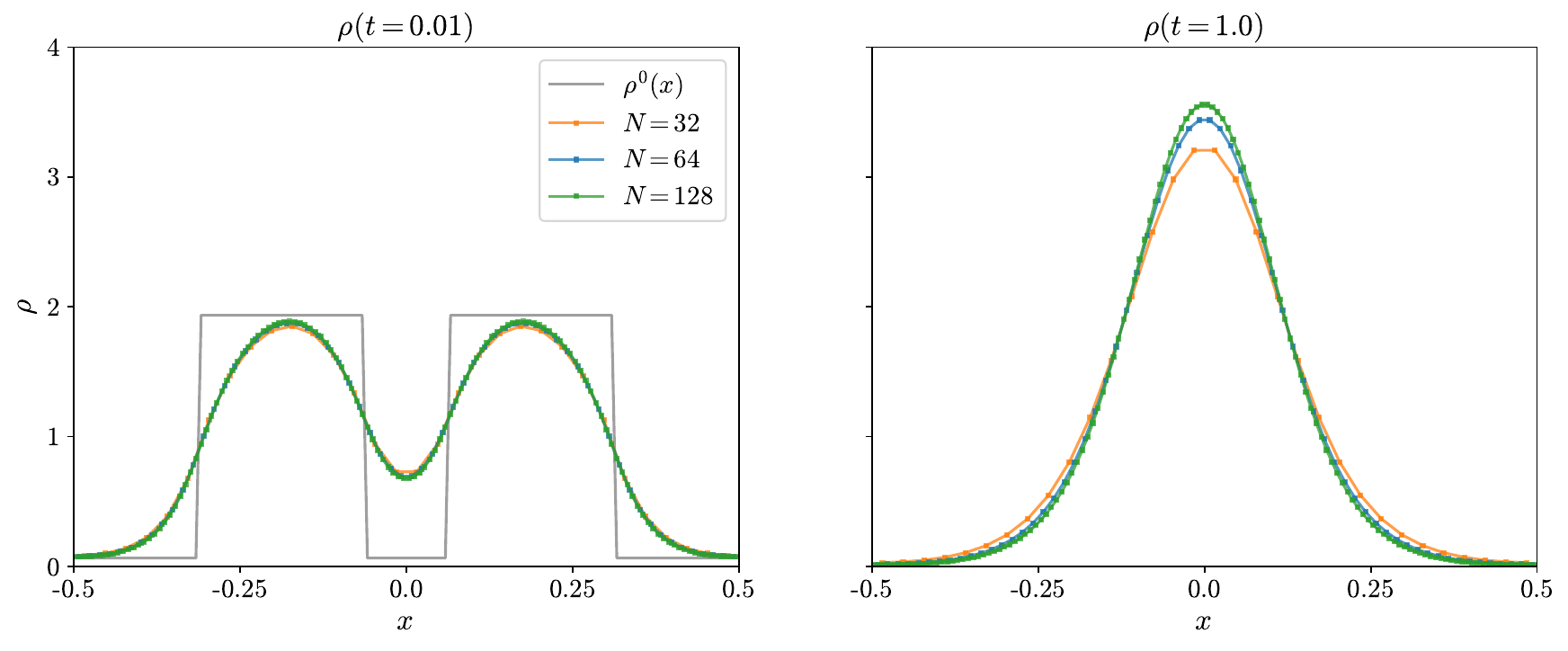}
    \caption{Spatial density $\rho(t,x)$ from \cref{eq:schemerho} at times $t = 0.1$ (left) and $t = 1.0$ (right) for different mesh sizes for $D_T=10^{-1},\gamma=500,\Pe=2,\alpha=1,B=B_0,\Delta t=10^{-2}$. 
    The initial condition (shown in grey on the left panel) is $f^0(x, \theta) = C \mathbf{1}_{|x\pm\frac{1}{8}|\leq\frac{1}{8}}$, with $C>0$ such that $\int f\dd x\dd\theta=1$. 
    }
    \label{fig:rhoev}
\end{figure}

\begin{figure}[thb]
    \centering
    \subfloat[$B_0$]{\includegraphics[width=0.85\textwidth]{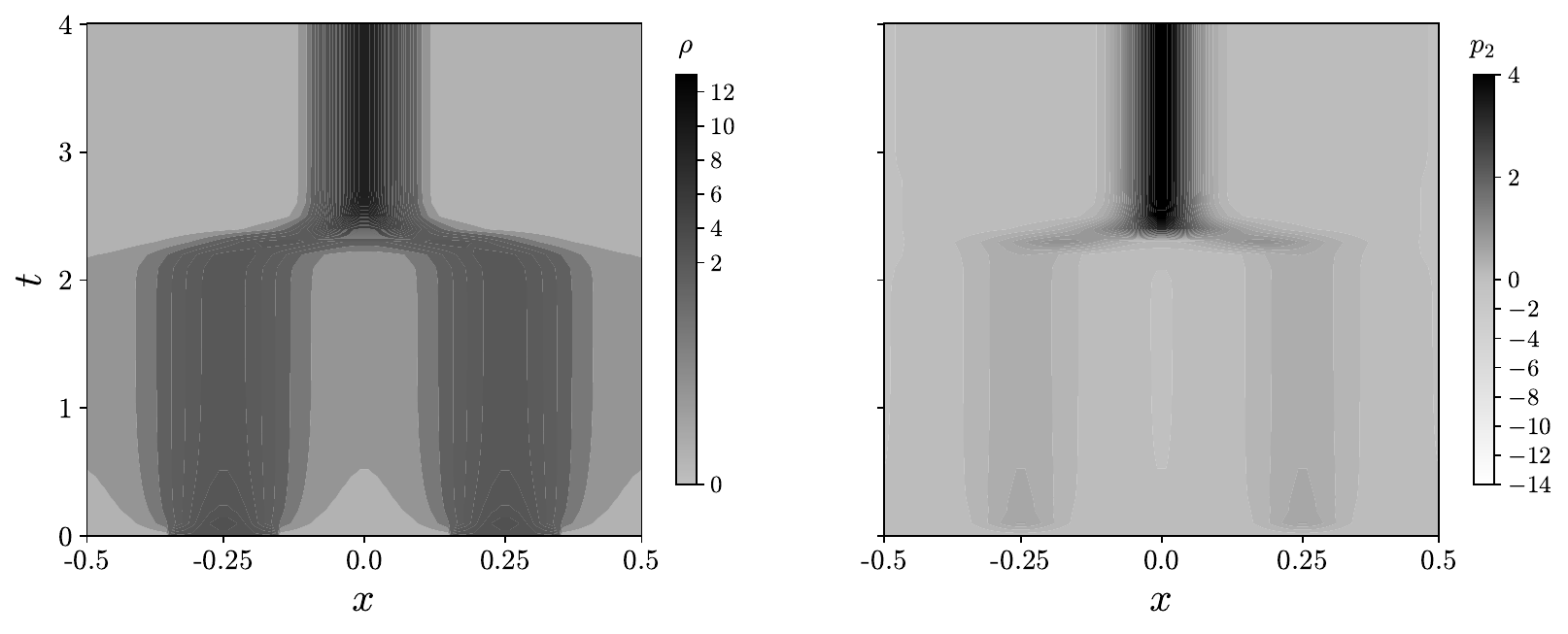}}\\
    \subfloat[$B_\lambda$]{\includegraphics[width=0.85\textwidth]{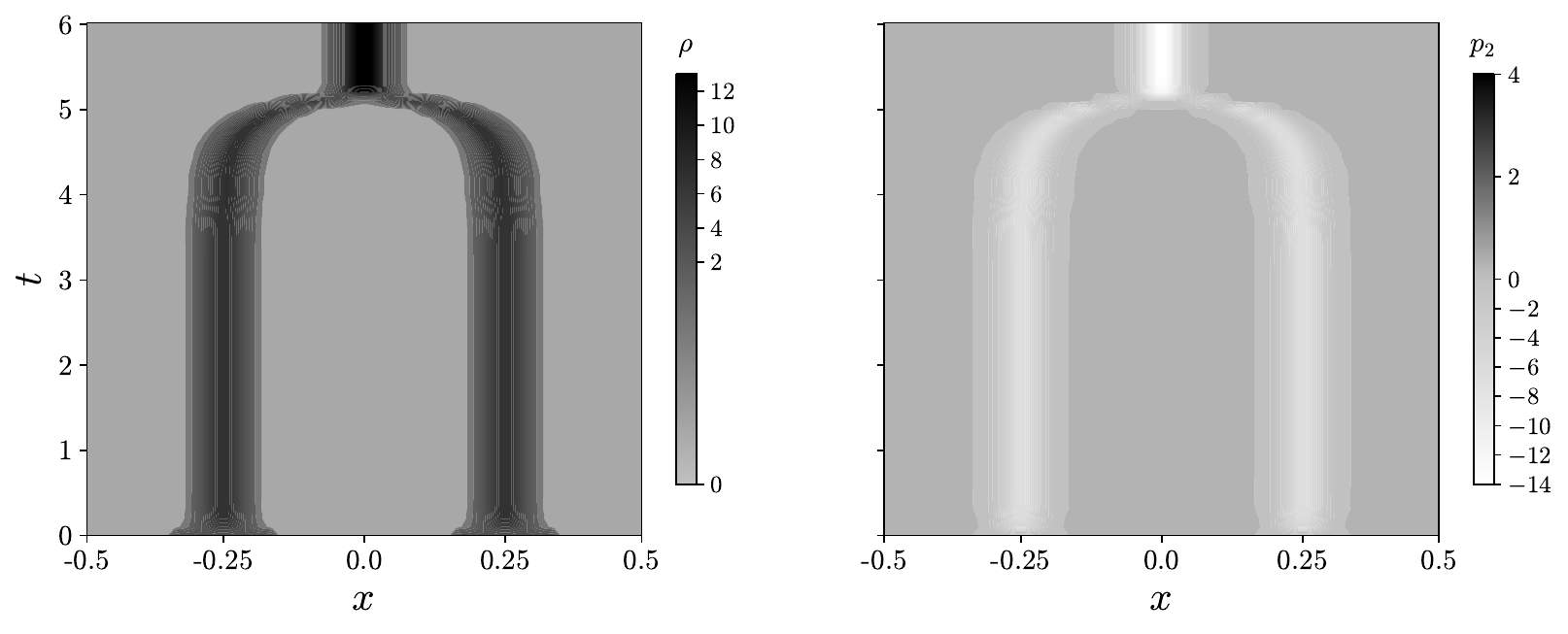}}
    \caption{Evolution of $\rho(t,x)$ from \cref{eq:schemerho} and $p_2(t,x) = \int_0^{2\pi}\cos(2\theta)f(t,x,\theta)\dd\theta$ for $B$-terms (a) $B_0$ \eqref{eq:B0_sch} and (b) $B_\lambda$ \eqref{eq:blambdadisc} with $\lambda=0.1$, for $D_T=10^{-1},\gamma=500,\Pe=2,\alpha=1,\Delta t=10^{-3},N=64$. The initial condition is $f^0(x, \theta) = C \mathbf{1}_{|x\pm\frac{1}{4}|\leq\frac{1}{8}}$, with $C>0$ such that $\int f^0\dd x\dd\theta=1$.}
    \label{fig:rhoev2}
\end{figure}
It is interesting to consider the effect the different interaction terms $B\in\{B_0,B_\lambda,B_\tau\}$ have on the evolution. Recall that $B_0$ and $B_\tau$ are the zeroth- and first-order Taylor expansions in $\lambda$ of $B_\lambda$. In \cite{bruna2024lane} it was observed that the model with $B_0$ can only display static aggregation, whereas $B_\lambda$ can lead to both static aggregations and dynamical laning. In \cref{fig:rhoev2}, the time-evolution of the spatial density $\rho(t,x)$ and  the second-order polarisation $p_2(t,x)=\int_0^{2\pi} \cos(2\theta) f(t,x,\theta)\dd\theta$ are shown for the interaction terms $B_0$ and $B_\lambda,\lambda=0.1$. We choose as initial data two bumps of equal mass,  $f^0(x,\theta)=C \mathbf{1}_{|x\pm\frac{1}{4}|\leq\frac{1}{4}}$, with $C$ such that $f^0$ is normalised. 
We observe that, for both interaction types $B_0$ and $B_\lambda$, the spatial density initially relaxes into two symmetric quasi-steady bumps, but eventually destabilises into one steady aggregate. The final spatial density is more localised for $B_\lambda$, that is, the bump is narrower and has a higher local maximum than that corresponding to $B_0$. While the spatial density looks qualitatively similar for both interaction models, the same is not true for the second-order polarisation. Here, we note a different sign between the models, with $p_2>0$ for $B_0$ and $p_2<0$ for $B_\lambda$. As noted in \cite{bruna2024lane}, the reason for this is that, for $\lambda=0$, the polarisation peaks at $\theta=0,\pi$, whereas when $\lambda$ increases, peaks emerge at $\theta=\pm \frac{\pi}{2}$, which correspond to a shift in the sign of $p_2$ (see \cite[Fig. 5 in]{bruna2024lane}). This shows that the (quasi-)steady states that the solutions converge to are different depending on the interaction term $B$, even though the spatial densities look qualitatively similar.

In \cref{fig:tauvlambda} we show the relative norm of the difference between the numerical solutions at final time $T$ and fixed mesh size for the interaction term $B_\lambda$ and $B_\tau$, with $\tau = \lambda$. We define the relative difference as 
\begin{equation}
    \mathrm{e}_{\lambda,\tau,L}(t)=\|f^\lambda(t)-f^\tau(t)\|_L/\|f^\lambda(t)\|_{L},
\end{equation}
where $f^\lambda$ is the solution for the interaction term \eqref{eq:blambdadisc} and $f^\tau$ for \eqref{eq:btaudisc} for the given mesh size, for any norm $L$. As $\lambda$ decreases, the difference between the two numerical solutions decreases. This is as expected, since $B_\lambda=B_{\tau=\lambda}+O(\lambda^2)$. This gives evidence that the dynamics for the interaction terms $B_\lambda$ and $B_\tau$ is mostly similar.

\begin{figure}[htb]
    \centering
    \includegraphics[width=0.48\linewidth]{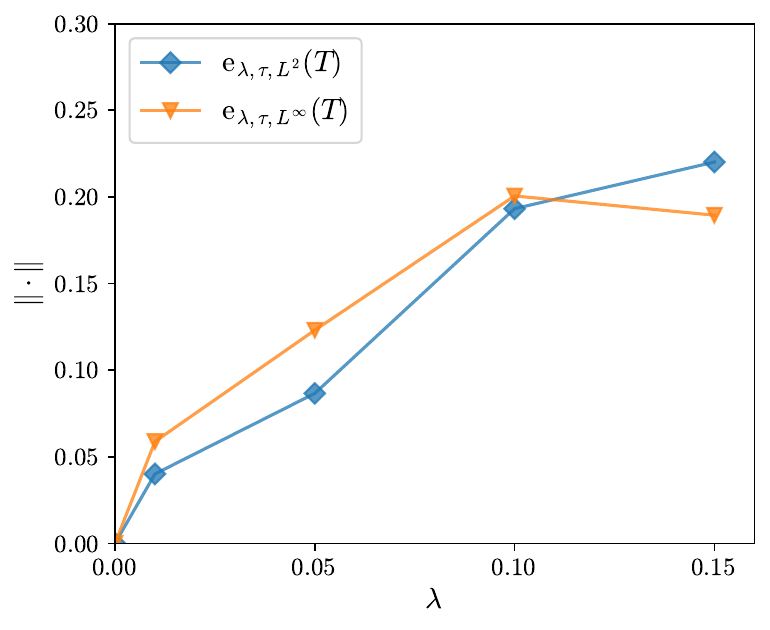}
    \caption{Relative norm of the difference between numerical solutions for $B_\lambda$ and $B_\tau$, $\lambda=\tau$, with $D_T=10^{-1},\gamma=500,\Pe=2,N=64,\Delta t=10^{-2}, T=1.0$. 
    The initial condition is $f^0(x, \theta) = C \mathbf{1}_{|x\pm\frac{1}{8}|\leq\frac{1}{8}}$, with $C>0$ such that $\int f\dd x\dd\theta=1$, as in \cref{fig:rhoev}.}
    \label{fig:tauvlambda}
\end{figure}

Next, we numerically test the order of convergence of the scheme \eqref{eq:scheme}. We define the relative error of the approximation $f_h$ with mesh size $h$ at time $t$ as
\begin{equation}\label{eq:reler}
    \mathrm{e}_{h,L}(t)=\|f_h(t)-f_{256^{-1}}(t)\|_L/\|f_{256^{-1}}(t)\|_L,
\end{equation}
for any norm $L$, with respect to the finest mesh with $N=256,h=2^{-8}$. In \cref{fig:compLs} shows the relative errors in $L^2$ and $L^\infty$ at the final time $T$ for the interaction term $B_0$. We use the same scheme and initial data as used for \cref{fig:rhoev}. 
This shows that the scheme is of order one in space, as can be expected for a linear finite volume discretisation. Since we are using an order one discretisation in time, we can expect that the full scheme is also of order one. In \cref{fig:compLslambda}, the relative error for the interaction term $B_\lambda$ with $\lambda=0.1$ is shown, again with the same scheme and initial data as in \cref{fig:rhoev} and the $B_\lambda$ term as in \cref{eq:blambdadisc}, and it shows that it complies with  order one convergence.
\begin{figure}[htb]
    \centering
    \subfloat[$B_0$]{\includegraphics[width=0.48\linewidth]{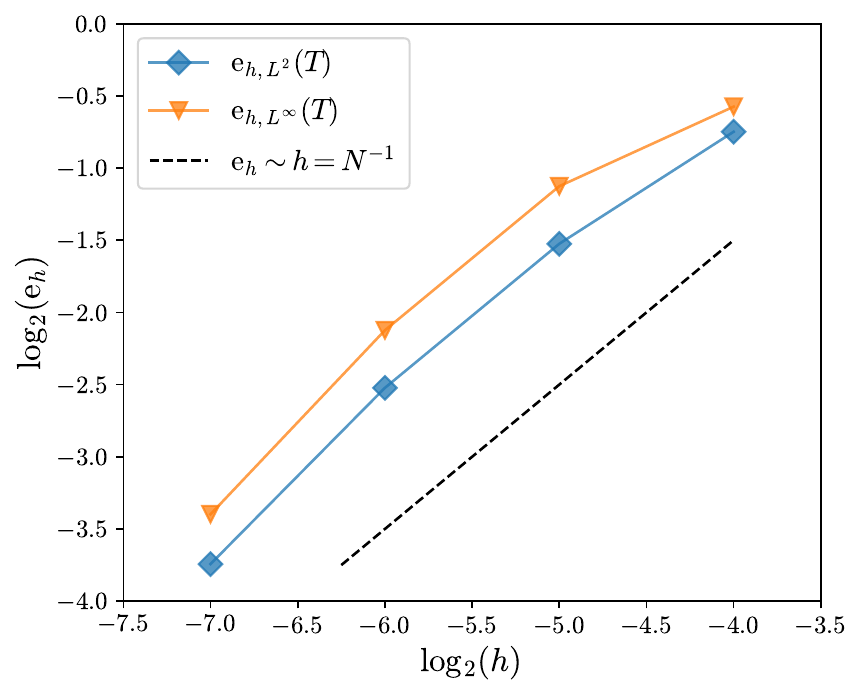}\label{fig:compLs}}
    \subfloat[$B_\lambda$]{\includegraphics[width=0.48\linewidth]{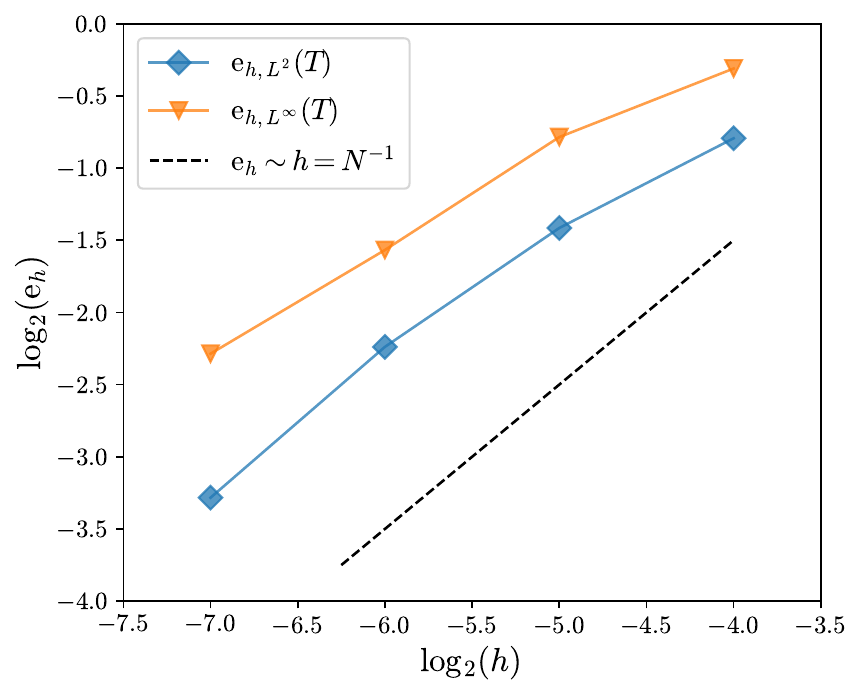}\label{fig:compLslambda}}
    \caption{Relative error \eqref{eq:reler} for numerical solutions to \eqref{eq:scheme} with varying mesh size, for (a) $B_0$ interaction \eqref{eq:B0_sch} and (b) $B_\lambda$ \eqref{eq:blambdadisc} with $\lambda=0.1$. The initial condition and the scheme are as in \cref{fig:rhoev}.
    Other parameters used: $T=1.0,D_T=10^{-1},\gamma=500,\Pe=2,\alpha=1,\Delta t=10^{-2}$.
    }
\end{figure}

\section*{Data availability statement}
The research code associated with this article are available in numerics\_ants, under the reference \url{https://github.com/odewit8/numerics_ants}.

\printbibliography

@Book{BCT2022,
 Editor = {Bellomo, Nicola and Carrillo, Jos{\'e} Antonio and Tadmor, Eitan},
 Title = {Active particles. {Volume} 3. {Advances} in theory, models, and applications},
 FSeries = {Modeling and Simulation in Science, Engineering and Technology},
 Series = {Model. Simul. Sci. Eng. Technol.},
 Year = {2022},
 Publisher = {Cham: Birkh{\"a}user}
}

@Book{BDT2019,
 Editor = {Bellomo, Nicola and Degond, Pierre and Tadmor, Eitan},
 Title = {Active particles, {Volume} 2},
subtitle = {{Advances} in theory, models, and applications},
 FSeries = {Modeling and Simulation in Science, Engineering and Technology},
 Series = {Model. Simul. Sci. Eng. Technol.},
 Year = {2019},
 Publisher = {Birkh{\"a}user Cham},
}

@Book{BDT2017,
 Editor = {Bellomo, Nicola and Degond, Pierre and Tadmor, Eitan},
 Title = {Active particles, {Volume} 1. {Advances} in theory, models, and applications},
 FSeries = {Modeling and Simulation in Science, Engineering and Technology},
 Series = {Model. Simul. Sci. Eng. Technol.},
 Year = {2017},
 Publisher = {Basel: Birkh{\"a}user/Springer}
}

@article{bailo2020convergence,
  title={Convergence of a Fully Discrete and Energy-Dissipating Finite-Volume Scheme for Aggregation-Diffusion Equations},
  author={Bailo, Rafael and Carrillo, Jos{\'e} A and Murakawa, Hideki and Schmidtchen, Markus},
  journal={Mathematical Models and Methods in Applied Sciences},
  volume={30},
  number={13},
  pages={2487--2522},
  year={2020},
  publisher={World Scientific}
}

@article{porretta2020note,
  title={A Note on the Sobolev and Gagliardo--Nirenberg Inequality when $p>N$},
  author={Porretta, Alessio},
  journal={Advanced Nonlinear Studies},
  volume={20},
  number={2},
  pages={361--371},
  year={2020},
  publisher={De Gruyter}
}

@article{perna2012individual,
  title={Individual rules for trail pattern formation in Argentine ants (Linepithema humile)},
  author={Perna, Andrea and Granovskiy, Boris and Garnier, Simon and Nicolis, Stamatios C and Lab{\'e}dan, Marjorie and Theraulaz, Guy and Fourcassi{\'e}, Vincent and Sumpter, David JT},
  journal={PLoS Computational Biology},
  volume={8},
  number={7},
  pages={e1002592},
  year={2012},
  publisher={Public Library of Science San Francisco, USA}
}

@article{bessemoulin2015discrete,
  title={On discrete functional inequalities for some finite volume schemes},
  author={Bessemoulin-Chatard, Marianne and Chainais-Hillairet, Claire and Filbet, Francis},
  journal={IMA Journal of Numerical Analysis},
  volume={35},
  number={3},
  pages={1125--1149},
  year={2015},
  publisher={Oxford University Press}
}

@article{bruna2024lane,
  title={Lane formation and aggregation spots in a model of ants},
  author={Bruna, Maria and Burger, Martin and de Wit, Oscar},
  journal={SIAM Journal on Applied Dynamical Systems},
  volume={24},
  number={1},
  pages={675--709},
  year={2025},
  publisher={SIAM}
}

@article{filbet2006finite,
  title={A finite volume scheme for the Patlak--Keller--Segel chemotaxis model},
  author={Filbet, Francis},
  journal={Numerische Mathematik},
  volume={104},
  pages={457--488},
  year={2006},
  publisher={Springer}
}

@article{zhou2017finite,
  title={Finite volume methods for a Keller--Segel system: discrete energy, error estimates and numerical blow-up analysis},
  author={Zhou, Guanyu and Saito, Norikazu},
  journal={Numerische Mathematik},
  volume={135},
  number={1},
  pages={265--311},
  year={2017},
  publisher={Springer}
}

@article{chainais2007asymptotic,
  title={Asymptotic behaviour of a finite-volume scheme for the transient drift-diffusion model},
  author={Chainais-Hillairet, Claire and Filbet, Francis},
  journal={IMA Journal of Numerical Analysis},
  volume={27},
  number={4},
  pages={689--716},
  year={2007},
  publisher={Oxford University Press}
}

@article{bertucci2024curvature,
  title={Curvature in chemotaxis: A model for ant trail pattern formation},
  author={Bertucci, Charles and Rakotomalala, Matthias and Tomasevic, Milica},
  journal={arXiv preprint arXiv:2408.13363},
  year={2024}
}

@article{kowalczyk2005preventing,
  title={Preventing blow-up in a chemotaxis model},
  author={Kowalczyk, Remigiusz},
  journal={Journal of Mathematical Analysis and Applications},
  volume={305},
  number={2},
  pages={566--588},
  year={2005},
  publisher={Elsevier},
doi={10.1016/j.jmaa.2004.12.009}
}

@article{bruna2022phase,
  title={Phase separation in systems of interacting active Brownian particles},
  author={Bruna, Maria and Burger, Martin and Esposito, Antonio and Schulz, Simon M},
  journal={SIAM Journal on Applied Mathematics},
  volume={82},
  number={4},
  pages={1635--1660},
  year={2022},
  publisher={SIAM}
}

@article{burger2023well,
  title={Well-posedness and stationary states for a crowded active Brownian system with size-exclusion},
  author={Burger, Martin and Schulz, Simon},
  journal={arXiv preprint arXiv:2309.17326},
  year={2023}
}

@article{eymard2000finite,
  title={Finite volume methods},
  author={Eymard, Robert and Gallou{\"e}t, Thierry and Herbin, Rapha{\`e}le},
  journal={Handbook of numerical analysis},
  volume={7},
  pages={713--1018},
  year={2000},
  publisher={Elsevier}
}

@article{albritton2023stabilizing,
  title={On the stabilizing effect of swimming in an active suspension},
  author={Albritton, Dallas and Ohm, Laurel},
  journal={SIAM Journal on Mathematical Analysis},
  volume={55},
  number={6},
  pages={6093--6132},
  year={2023},
  publisher={SIAM}
}

@article{briant2022cauchy,
  title={Cauchy theory for general kinetic Vicsek models in collective dynamics and mean-field limit approximations},
  author={Briant, Marc and Diez, Antoine and Merino-Aceituno, Sara},
  journal={SIAM Journal on Mathematical Analysis},
  volume={54},
  number={1},
  pages={1131--1168},
  year={2022},
  publisher={SIAM}
}

@article{carrillo2023invariance,
  title={An invariance principle for gradient flows in the space of probability measures},
  author={Carrillo, Jos{\'e} A and Gvalani, Rishabh S and Wu, Jeremy S-H},
  journal={Journal of Differential Equations},
  volume={345},
  pages={233--284},
  year={2023},
  publisher={Elsevier}
}

@article{ohm2022weakly,
  title={Weakly nonlinear analysis of pattern formation in active suspensions},
  author={Ohm, Laurel and Shelley, Michael J},
  journal={Journal of Fluid Mechanics},
  volume={942},
  pages={A53},
  year={2022},
  publisher={Cambridge University Press}
}

@article{degond2013macroscopic,
  title={Macroscopic limits and phase transition in a system of self-propelled particles},
  author={Degond, Pierre and Frouvelle, Amic and Liu, Jian-Guo},
  journal={Journal of Nonlinear Science},
  volume={23},
  pages={427--456},
  year={2013},
  publisher={Springer}
}

@article{burger2016lane,
  title={Lane formation by side-stepping},
  author={Burger, Martin and Hittmeir, Sabine and Ranetbauer, Helene and Wolfram, Marie-Therese},
  journal={SIAM Journal on Mathematical Analysis},
  volume={48},
  number={2},
  pages={981--1005},
  year={2016},
  publisher={SIAM}
}

@article{carrillo2018zoology,
  title={Zoology of a nonlocal cross-diffusion model for two species},
  author={Carrillo, Jos{\'e} A and Huang, Yanghong and Schmidtchen, Markus},
  journal={SIAM Journal on Applied Mathematics},
  volume={78},
  number={2},
  pages={1078--1104},
  year={2018},
  publisher={SIAM}
}

@book{evans2010partial,
  title={Partial Differential Equations},
  author={Evans, Lawrence C},
edition={2},
  year={2010},
  publisher={American Mathematical Society}
}

@article{hemelrijk2012schools,
  title={Schools of fish and flocks of birds: their shape and internal structure by self-organization},
  author={Hemelrijk, Charlotte K and Hildenbrandt, Hanno},
  journal={Interface Focus},
  volume={2},
  number={6},
  pages={726--737},
  year={2012},
  publisher={The Royal Society}
}

@article{painter1999stripe,
  title={Stripe formation in juvenile Pomacanthus explained by a generalized Turing mechanism with chemotaxis},
  author={Painter, Kevin J and Maini, Philip K and Othmer, Hans G},
  journal={Proceedings of the National Academy of Sciences},
  volume={96},
  number={10},
  pages={5549--5554},
  year={1999},
  publisher={National Acad Sciences}
}

@article{alert2022active,
  title={Active turbulence},
  author={Alert, Ricard and Casademunt, Jaume and Joanny, Jean-Fran{\c{c}}ois},
  journal={Annual Review of Condensed Matter Physics},
  volume={13},
  number={1},
  pages={143--170},
  year={2022},
  publisher={Annual Reviews}
}

@article{bacik2023lane,
  title={Lane nucleation in complex active flows},
  author={Bacik, Karol A and Bacik, Bogdan S and Rogers, Tim},
  journal={Science},
  volume={379},
  number={6635},
  pages={923--928},
  year={2023},
  publisher={American Association for the Advancement of Science}
}

@book{brezis2010functional,
  title={Functional {A}nalysis, Sobolev {S}paces and Partial Differential Equations},
  author={Brezis, Haim},
  year={2010},
  edition={1},
  publisher={Springer},
address={New York, NY, USA},
    url={https://doi.org/10.1007/978-0-387-70914-7}
}

@article{simon1986compact,
  title={Compact sets in the space $L^p(O, T; B)$},
  author={Simon, Jacques},
  journal={Annali di Matematica pura ed applicata},
  volume={146},
  pages={65--96},
  year={1986},
  publisher={Springer}
}

@article{chainais2003finite,
  title={Finite volume scheme for multi-dimensional drift-diffusion equations and convergence analysis},
  author={Chainais-Hillairet, Claire and Liu, Jian-Guo and Peng, Yue-Jun},
  journal={ESAIM: Mathematical Modelling and Numerical Analysis},
  volume={37},
  number={2},
  pages={319--338},
  year={2003},
  publisher={EDP Sciences}
}

@misc{odewit8github,
  author = {de Wit, Oscar},
  title = {numerics ants},
  year = {2025},
  url = {https://github.com/odewit8/numerics_ants},
  note = {Last accessed 25 March 2025},
}

\appendix

\section{Calculations for look-ahead term}\label{sec:appendix}
\begin{mydef}[$B_\lambda$-term via interpolation]\label{def:discblambda}
    We define the $B_\lambda$-term as 
    \begin{equation}
    B_\lambda[c^n]_{i,j,k+1/2}=\mathbf{n}(\theta_{k+1/2})\cdot\discgrad c_{i,j,k+1/2}^n,
\end{equation}
where the discrete gradient with triple indices is accordingly defined as
\begin{equation}
    \discgrad c_{i,j,k+1/2}=\frac{1}{2}\begin{pmatrix}
        \frac{1}{\Delta x}[c_{i+1,j,k+1/2}-c_{i-1,j,k+1/2})\\
        \frac{1}{\Delta y}(c_{i,j+1,k+1/2}-c_{i,j-1,k+1/2})
    \end{pmatrix},
\end{equation}
as before and where the term $c_{i,j,k+1/2}$ is defined by nearest-neighbour interpolation to approximate $c$ at the point $\xx_{i,j}+\lambda\mathbf{e}(\theta_{k+1/2})$, so that $c_{i,j,k+1/2}\approx c(\xx_{i,j}+\lambda\mathbf{e}(\theta_{k+1/2}))$. That is,  we define
\begin{equation}
c_{i,j,k+1/2}=c_{i(k+1/2),j(k+1/2)} \ \mathrm{if} \ \mathbf{x}_{i,j}+\lambda\mathbf{e}(\theta_{k+1/2})\in C_{i(k+1/2),j(k+1/2)},
\end{equation}
so that $i(k+1/2)$ and $j(k+1/2)$ denote the $i$- and $j$-indices of the cell in which the point $\mathbf{x}_{i,j}+\lambda\mathbf{e}(\theta_{k+1/2})$ lies.
See also \cref{fig:cijk} for an illustration of the points involved in the nearest-neighbour interpolation.
\end{mydef}
\begin{figure}[htb]

\centering
\begin{tikzpicture}
    \draw[step=1cm,gray!50,thin] (-1,-1) grid (5,5);
    
    \node[fill=black,circle,inner sep=1pt] (ij) at (0.5,0.5) {};
    \node[below] at (ij) {\(i,j\)};
    
    \node[fill=black,circle,inner sep=1pt] (k_plus_half) at (2.5,3.5) {};
    \node[above] at (k_plus_half) {\(i(k+1/2),j(k+1/2)\)};
    
    \node[fill=orange,circle,inner sep=1pt] (lambda1) at (2.8,3.2) {};
    \node[below] at (lambda1) {\(\xx_{i,j} + \lambda \mathbf{e}(\theta_{k+1/2})\)};
    
    \node[fill=orange,circle,inner sep=1pt] (lambda2) at (3.75,1.752) {};
    \node[below] at (lambda2) {\(\xx_{i,j} + \lambda \mathbf{e}(\theta_{k-1/2})\)};
    
\end{tikzpicture}
\caption{Nearest-neighbour interpolation for $B_\lambda$.}
\label{fig:cijk}
\end{figure}
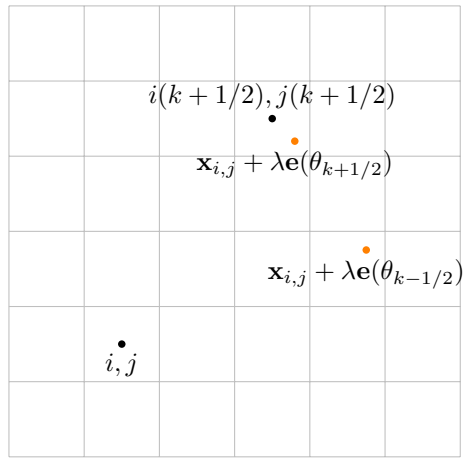
\begin{proof}[Proof of $B_\lambda$ interaction.]
We have
    \begin{equation}
        \dd_\theta B_\lambda[c]_{i,j,k}=\frac{1}{\Delta \theta}\left(\mathbf{n}(\theta_{k+1/2})\cdot\discgrad c_{i,j,k+1/2}-\mathbf{n}(\theta_{k-1/2})\cdot\discgrad c_{i,j,k-1/2}\right).
    \end{equation}
By adding and subtracting a term with mixed signs in the $k$-index we can also write this as 
\begin{equation}
    \begin{split}
        \dd_\theta B_\lambda[c]_{i,j,k}=&\frac{1}{\Delta \theta}(\mathbf{n}(\theta_{k+1/2})-\mathbf{n}(\theta_{k-1/2}))\cdot\discgrad c_{i,j,k+1/2}\\+&\frac{1}{\Delta\theta}\mathbf{n}(\theta_{k-1/2})\cdot(\discgrad c_{i,j,k+1/2}-\discgrad c_{i,j,k-1/2})\\
        =:& T_1+T_2.
    \end{split}
\end{equation}
The idea of the proof is to bound $T_1$ and $T_2$ as the terms that would appear by applying the chain rule and product rule at the continuum level,
\begin{equation}
    \partial_\theta (\mathbf{n}_\theta\cdot\nabla c(\xx+\lambda\mathbf{e}_\theta))=\underbrace{-\mathbf{e}_\theta\cdot\nabla c(\xx+\lambda\mathbf{e}_\theta))}_{\approx T_1}+\underbrace{\lambda\mathbf{n}_\theta\cdot D^2 c(\xx+\lambda\mathbf{e}_\theta))\mathbf{n}_\theta}_{\approx T_2}.
\end{equation}
For $T_1$, we first note that
\begin{equation}
 \frac{c_{i+1,j,k+1/2}-c_{i-1,j,k+1/2}}{2\Delta x}=\frac{c_{i(k+1/2)+1,j(k+1/2)}-c_{i(k+1/2)-1,j(k+1/2)}}{2\Delta x},
\end{equation}
so that we have
\begin{equation}
    \discgrad_x c_{i,j,k+1/2}=\discgrad_x c_{i(k+1/2),j(k+1/2)},
\end{equation}
and the same applies to the $y$-component. Therefore, we can bound $T_1$ as
\begin{equation}
    |T_1|\leq \max_{\theta\in[0,2\pi]}|\mathbf{e}(\theta)||\discgrad c_{i(k+1/2),j(k+1/2)}|.
\end{equation}
For $T_2$, we split the $x$- and $y$-components of the inner product, so that we may write
\begin{equation}
\begin{split}
    T_2=&-\tfrac{1}{\Delta\theta}\sin(\theta_{k-1/2})(\discgrad_x c_{i(k+1/2),j(k+1/2)}-\discgrad_x c_{i(k-1/2),j(k-1/2)})\\&+\tfrac{1}{\Delta\theta}\cos(\theta_{k-1/2})(\discgrad_y c_{i(k+1/2),j(k+1/2)}-\discgrad_y c_{i(k-1/2),j(k-1/2)})\\
    =:&T_2^x+T^y_2.
\end{split}
\end{equation}
We note that, for the discrete derivatives in the $x$-component, after adding a term with mixed signs for the $k$-index,
\begin{equation}
\begin{split}
    &\discgrad_x c_{i(k+1/2),j(k+1/2)}-\discgrad_x c_{i(k-1/2),j(k-1/2)}\\
    = \ &(\discgrad_x c_{i(k+1/2),j(k+1/2)}-\discgrad_x c_{i(k-1/2),j(k+1/2)})
    +(\discgrad_x c_{i(k-1/2),j(k+1/2)}-\discgrad_x c_{i(k-1/2),j(k-1/2)})\\
    =: \ &T_{2,1}^x+T_{2,2}^x.
\end{split}
\end{equation}
Now, focusing on $T^x_{2,1}$, we can add a telescoping sum ranging in the $i$-index from $i(k-1/2)$ to $i(k+1/2)$, so that
\begin{equation}
\begin{split}
    T_{2,1}^x=&-\discgrad_x c_{i(k-1/2),j(k+1/2)}+\sum_{r=i(k-1/2)+1}^{i(k+1/2)-1}[\discgrad_x c_{r,j(k+1/2)}-\discgrad_x c_{r,j(k+1/2)}]+\discgrad_x c_{i(k+1/2),j(k+1/2)}\\
    =&\sum_{r=i(k-1/2)}^{i(k+1/2)-1}[\discgrad_x c_{r+1,j(k+1/2)}-\discgrad_x c_{r,j(k+1/2)}].
\end{split}
\end{equation}
By adding and subtracting at each summand the term $\dd_x c_{r+1/2,j(k+1/2)}$ (as defined on the dual mesh) we get 
\begin{equation}
\begin{split}
    T_{2,1}^x=\sum_{r=i(k-1/2)}^{i(k+1/2)-1}\frac{1}{2\Delta x}\Big[&(c_{r+2,j(k+1/2)}-c_{r,j(k+1/2)})-2(c_{r+1,j(k+1/2)}-c_{r,j(k+1/2)})\\& + 2(c_{r+1,j(k+1/2)}-c_{r,j(k+1/2)})-(c_{r+1,j(k+1/2)}-c_{r-1,j(k+1/2)})\Big]\\
    = \sum_{r=i(k-1/2)}^{i(k+1/2)-1}\frac{\Delta x}{2}(&\dd_x^2 c_{r+1,j(k+1/2)}+\dd_x^2 c_{r,j(k+1/2)}).
\end{split}
\end{equation}
For $T_{2,2}^x$, we do a similar calculation. By adding a telescoping sum we obtain
\begin{equation}
    \begin{split}
        T_{2,2}^x=\sum_{r=j(k-1/2)}^{j(k+1/2)-1}[\discgrad_x c_{i(k-1/2),r+1}-\discgrad_x c_{i(k-1/2),r}],
    \end{split}
\end{equation}
and by adding and subtracting $\frac{1}{2\Delta x}[c_{i(k-1/2),r+1}-c_{i(k-1/2),r}]$ for each summand we can write 
\begin{equation}
\begin{split}
    T_{2,2}^x=\sum_{r=j(k-1/2)}^{j(k+1/2)-1}\frac{\Delta y}{2}[\dd_x\dd_yc_{i(k-1/2)-1/2,r+1/2}+\dd_x\dd_yc_{i(k-1/2)+1/2,r+1/2}],
\end{split}
\end{equation}
recalling the notation 
\begin{equation}
    \dd_x\dd_yc_{i+1/2,j+1/2}=\frac{1}{\Delta\xx}[c_{i,j}-c_{i+1,j}-c_{i,j+1}+c_{i+1,j+1}].
\end{equation}
We now recall the definitions $C_{\Delta x}=\Delta\theta/\Delta x,C_{\Delta y}=\Delta\theta/\Delta y$, so that by inserting $T_{2,1}^x$ and $T_{2,2}^x$ from above into $T^x_2$ we are lead to
\begin{equation}
\begin{split}
    T_{2}^x=&\frac{-\sin(\theta_{k-1/2})}{2 C_{\Delta x}}\sum_{r=i(k-1/2)}^{i(k+1/2)-1}[\dd_x^2 c_{r+1,j(k+1/2)}+\dd_x^2 c_{r,j(k+1/2)}]\\+&\frac{-\sin(\theta_{k-1/2})}{2 C_{\Delta y}}\sum_{r=j(k-1/2)}^{j(k+1/2)-1}[\dd_x\dd_yc_{i(k-1/2)-1/2,r+1/2}+\dd_x\dd_yc_{i(k-1/2)+1/2,r+1/2}].
\end{split}
\end{equation}
We can use a similar argument for $T_{2}^y$, so that
\begin{equation}
\begin{split}
    T_{2}^y=&\frac{\cos(\theta_{k-1/2})}{2 C_{\Delta x}}\sum_{r=i(k-1/2)}^{i(k+1/2)-1}[\dd_x\dd_yc_{r+1/2,j(k-1/2)-1/2}+\dd_x\dd_yc_{r+1/2,j(k-1/2)+1/2}]\\+&\frac{\cos(\theta_{k-1/2})}{2 C_{\Delta y}}\sum_{r=j(k-1/2)}^{j(k+1/2)-1}[\dd_y^2 c_{i(k+1/2),r+1}+\dd_y^2 c_{i(k+1/2),r}],
\end{split}
\end{equation}
Putting everything together, we get, using Jensen's inequality,
\begin{equation}
\begin{split}
    &\|\dd_\theta B[c_h]\|_{L^2}^2\leq \sum_{i,j,k}|\discgrad c_{i(k+1/2),j(k+1/2)}|^2\Delta\xi\\&+\max\{2\pi\lambda,\tfrac{1}{2\varepsilon_C}\} \sum_{i,j,k}[|\dd_x^2 c_{i(k+1/2),j(k+1/2)}|^2+2|\dd_x\dd_y c_{i(k+1/2)+1/2,j(k+1/2)+1/2}|^2+|\dd_y^2 c_{i(k+1/2),j(k+1/2)}|^2]\Delta\xi
    \\&\leq(2\pi)[\|\dd_\xx c_h\|_{L^2}^2+\max\{2\pi\lambda,\tfrac{1}{2\varepsilon_C}\} \|\dd_\xx^2 c_h\|_{L^2}^2].
\end{split}
\end{equation}
Here we used that $|i(k+1/2)-1-i(k-1/2)|\leq \max\{2\pi \lambda C_{\Delta x},1\},|j(k+1/2)-1-j(k-1/2)|\leq \max\{2\pi \lambda C_{\Delta y},1\}$ and that $C_{\Delta x}$ and $C_{\Delta y}$ are uniformly bounded below strictly from zero by $\varepsilon_C>0$.

We also note that it is clear that, for any grid function $(g_{i,j})_{(i,j)}$, $\sum_{i,j}g_{i(k),j(k)}=\sum_{i,j}g_{i,j}$, since the $k$-index shift induces only a translation of the grid points.

This concludes the proof for the $B_\lambda$-term.
\end{proof}
\section{Algebra for the Alikakos lemma}\label{app:alikalgebra}
\begin{proof}[Algebra for \eqref{eq:AlikRhoAlgebra}]
    First we note that the discrete Gagliardo--Nirenberg inequality in dimension two implies 
    \begin{equation}
        \|u\|_{L^2}^2\leq \frac{1}{\kappa}C^2_{GN}\|u\|_{L^1}^2+\kappa\|u\|_{1,2}^2,
    \end{equation}
    for any $0<\kappa<1$. This can be rewritten as
    \begin{equation}\label{eq:kapparho}
        -\frac{\kappa}{1-\kappa}|u|_{1,2}^2\leq-\|u\|_{L^2}^2 +\frac{1}{\kappa(1-\kappa)}C_{GN}^2\|u\|_{L^1}^2.
    \end{equation}
    We now let $\kappa$ be such that $2^{-qk}=\frac{\kappa}{1-\kappa}$. We let $\varepsilon$ be such that $D_T-\frac{\varepsilon}{2}\Pe^2=\frac{1}{2}$, and define $p=2^k-1$,$a_k=\frac{2}{\varepsilon}(2^k-1)2^k=\frac{2}{\varepsilon}p(p+1)$, $\epsilon_k=\frac{1}{2^{qk}}$ and $c_k=2^{\beta k}$, so that $\beta\geq 0$ is such that $C_{GN}^2\frac{1}{\kappa(1-\kappa)}\leq 2^{\beta k}$ for all integers $k\geq 1$. (The latter is true for $\beta$ sufficiently larger than $q$, but we still have to choose $q$.) Then, \eqref{eq:kapparho} can be written as
    \begin{equation}\label{eq:rhoalik1}
        -\epsilon_k|u|_{1,2}^2\leq -\|u\|_{L^2}^2+c_k\|u\|_{L^1}^2.
    \end{equation}
    We now require $-\frac{2p}{p+1}\leq -(a_k+\epsilon_k)\epsilon_k$, which can be written as 
    \begin{equation}
        -\frac{2(2^k-1)}{2^k}\leq -\left(\frac{2}{\varepsilon}(2^k-1)2^k+\frac{1}{2^{qk}}\right)\frac{1}{2^{qk}}.
    \end{equation}
    This inequality holds for all integers $k\geq 1$ for $q\geq 0$ large enough.

    Hence, multiplying \eqref{eq:rhoalik1} by $(a_k+\epsilon_k)$ we get 
    \begin{equation}\label{eq:rhoalik2}
        -\frac{2p}{p+1}|(\rho_h)^{\frac{p+1}{2}}|_{1,2}^2\leq-(a_k+\epsilon_k)\epsilon_k|(\rho_h)^{\frac{p+1}{2}}|_{1,2}^2\leq-(a_k+\epsilon_k)\|(\rho_h)^{\frac{p+1}{2}}\|_{L^2}^2+(a_k+\epsilon_k)c_k \|(\rho_h)^{\frac{p+1}{2}}\|_{L^1}^2.
    \end{equation}
    This gives the algebra to derive \eqref{eq:AlikRhoAlgebra}.
\end{proof}
\begin{proof}[Algebra for \eqref{eq:AlikFAlgebra}]
    The discrete Gagliardo--Nirenberg inequality in dimension three implies 
    \begin{equation}
        \|u\|_{L^2}^2\leq \frac{1}{\kappa^{\frac{5}{2}}} C_{GN}^3\|u\|_{L^1}^2+\kappa^{\frac{5}{3}}\|u\|_{1,2}^2.
    \end{equation}
    This can be rewritten as 
    \begin{equation}
        -\frac{\kappa^{\frac{5}{3}}}{1-\kappa^\frac{5}{3}}|u|_{1,2}^2\leq -\|u\|_{L^2}^2+\frac{C_{GN}^3}{\kappa^{\frac{5}{2}}(1-\kappa^{\frac{5}{3}})}\|u\|_{L^1}^2.
    \end{equation}
    We choose $q$ and $\kappa$ such that $2^{-qk}=\frac{\kappa^{\frac{5}{3}}}{1-\kappa^\frac{5}{3}}$. From here on, we can do the same algebra as for \eqref{eq:rhoalik2} to conclude that 
    \begin{equation}
        -\frac{2p}{p+1}|(f_h)^{\frac{p+1}{2}}|_{1,2}^2\leq-(a_k+\epsilon_k)\epsilon_k|(f_h)^{\frac{p+1}{2}}|_{1,2}^2\leq-(a_k+\epsilon_k)\|(f_h)^{\frac{p+1}{2}}\|_{L^2}^2+(a_k+\epsilon_k)c_k \|(f_h)^{\frac{p+1}{2}}\|_{L^1}^2.
    \end{equation}
\end{proof}
\section{Heatmap for 3D simulation}\label{app:fig3d}
\begin{figure}[htb]
    \centering
    \includegraphics[width=0.98\textwidth]{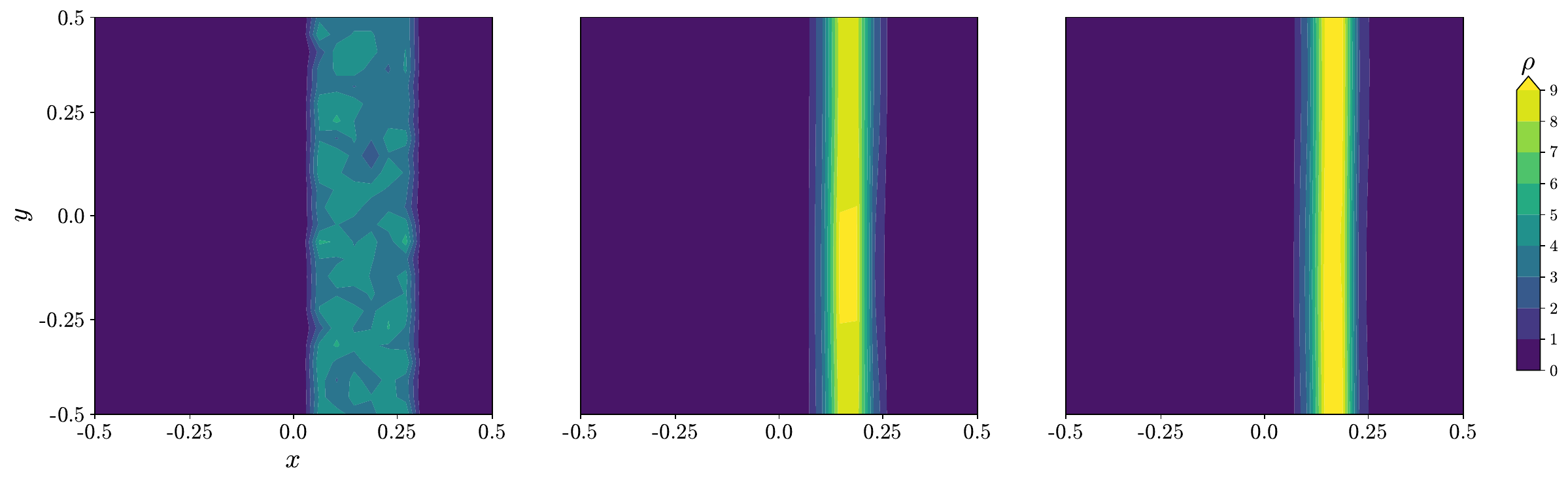}
    \caption{Evolution of $\rho(t,x,y)$ \eqref{eq:scheme-rho} with \cref{eq:btaudisc}for times $t=0,0.31,0.41$ from left to right and $D_T=10^{-2},\gamma=250,\Pe=3,\tau=0.5,\alpha=1,\Delta t=10^{-3},N_x=N_y=24,=N_\theta=16$.}\label{fig:btau}
\end{figure}
\end{document}